\documentclass[12pt]{article}
 \pdfoutput=1

\usepackage{ifthen}
\usepackage{amssymb}
\usepackage{amsmath}
\usepackage{MnSymbol}
\usepackage{stmaryrd}
\usepackage{mathrsfs}
\usepackage{theorem}
\usepackage{enumitem}
\usepackage{color}
\usepackage{graphicx}
\usepackage{tikz}
\usetikzlibrary{arrows,matrix}
\usepackage{xparse}
\usepackage[backend = biber]{biblatex}
\usepackage{authblk}
\usepackage{verbatim}
\usepackage[section]{algorithm}
\usepackage{algorithmic}
\usepackage{lineno}

	{
		\theorembodyfont{\itshape}

		\newtheorem{theorem}[algorithm]{Theorem}
		\newtheorem{lemma}[algorithm]{Lemma}
		\newtheorem{proposition}[algorithm]{Proposition}
		\newtheorem{corollary}[algorithm]{Corollary}
		
	}

	{
		\theorembodyfont{\rmfamily}
		\newtheorem{definition}[algorithm]{Definition}
		
		\newtheorem{remark}[algorithm]{Remark}
		\newtheorem{example}[algorithm]{Example}
		
	}

	\newenvironment{proof}{
		\goodbreak\par
		\textit{Proof.}%
	}{%
		\nopagebreak
		\hfill{\vrule width 1ex height 1ex depth 0ex}
		\medskip
		\goodbreak
	}

	\newcommand{\sizedescriptor}[2]
	{
		\ifthenelse{\equal{#1}{0}}{}{
		\ifthenelse{\equal{#1}{1}}{\big}{
		\ifthenelse{\equal{#1}{2}}{\Big}{
		\ifthenelse{\equal{#1}{3}}{\bigg}{
		\ifthenelse{\equal{#1}{4}}{\Bigg}{
		#2}}}}}
	}

	\newcommand{\ie}[1][~]{i.e.{#1}}

	\newcommand{\cf}[1][~]{cf.{#1}}

	\newcommand{\df}[1]{\emph{#1}}
	\newcommand{\ism}{\cong}
	
	\newcommand{\dfeq}{\mathrel{\mathop:}=}
	\newcommand{\dfeqrev}{=\mathrel{\mathop:}}
	\newcommand{\sepeq}{\ =\ }
	\newcommand{\sepdfeq}{\quad\dfeq\quad}
	
	\newcommand{\impl}{\Rightarrow}
	\newcommand{\revimpl}{\Leftarrow}

	\newcommand{\im}{\mathrm{im}}

	\newcommand{\xall}[3]{\forall\, #1 \,{\in}\, #2\,.\,#3}
	\newcommand{\xsome}[3]{\exists\, #1 \,{\in}\, #2\,.\,#3}

	\DeclareDocumentCommand{\set}{O{auto} m G{\empty}}{ \sizedescriptor{#1}{\left} \{{#2} \ifthenelse{\equal{#3}{}}{}{ \; \sizedescriptor{#1}{\middle}| \; {#3}} \sizedescriptor{#1}{\right}\}}

	\newcommand{\inhpst}{\mathscr{P}_+}

	\newcommand{\NN}{\mathbb{N}}
	\newcommand{\ZZ}{\mathbb{Z}}
	\newcommand{\QQ}{\mathbb{Q}}
	\newcommand{\RR}{\mathbb{R}}
	
	\newcommand{\intcc}[3][\RR]{{#1}_{[#2, #3]}}
	
	\newcommand{\intco}[3][\RR]{{#1}_{[#2, #3)}}

	\DeclareDocumentCommand{\ch}{O{*} m G{\empty}}{C_{#1}\left(#2\ifthenelse{\equal{#3}{}}{}{; {#3}}\right)}  
	\DeclareDocumentCommand{\cy}{O{*} m G{\empty}}{Z_{#1}\left(#2\ifthenelse{\equal{#3}{}}{}{; {#3}}\right)}  
	\DeclareDocumentCommand{\bd}{O{*} m G{\empty}}{B_{#1}\left(#2\ifthenelse{\equal{#3}{}}{}{; {#3}}\right)}  
	\DeclareDocumentCommand{\hm}{O{*} m G{\empty}}{H_{#1}\left(#2\ifthenelse{\equal{#3}{}}{}{; {#3}}\right)}  
	\newcommand{\field}{\mathbb{F}}
	\newcommand{\linspan}[1]{\left\langle{#1}\right\rangle}  
	\newcommand{\ec}{\chi}  
	\newcommand{\cx}{\mathcal{Q}}  
	\newcommand{\scx}{\mathcal{S}}  
	\newcommand{\skl}[2][d]{#2^{(#1)}}  
	\newcommand{\pc}{\mathcal{C}}  
	\newcommand{\wsx}{\pc_w}  
	\newcommand{\rwsx}[1][\alpha]{\pc_{w \leq {#1}}}  
	\newcommand{\dsk}[1][\alpha]{\skl{\rwsx[#1]}}  
	\newcommand{\tw}{\mathrm{tw}}  
	\newcommand{\stage}[1][\alpha]{Q(\pc; #1)}  
	\newcommand{\filt}{\big(\stage\big)_{\alpha \in \RR}}  
	\newcommand{\mst}{\mathtt{MST}^{(d)}}  
	\newcommand{\rmst}[1][\alpha]{\mathtt{MST}_{#1}^{(d)}}  
	\newcommand{\armst}[1][\alpha]{\widetilde{\mathtt{MST}}_{#1}^{(d)}}  
	\newcommand{\hps}{\mathtt{HoPeS}^{(d)}}  
	\newcommand{\rhps}[1][\alpha]{\mathtt{HoPeS}^{(d)}_{#1}}  
	\newcommand{\birth}{\textrm{birth}}
	\newcommand{\death}{\textrm{death}}
	\newcommand{\crit}[1][\alpha]{\widetilde{\mathcal{K}}_{#1}}  
	\newcommand{\critl}[1][\alpha]{\mathcal{K}_{#1}}  
	\newcommand{\eR}{\overline{\RR}}  
	\newcommand{\inford}{\sqsubseteq}  
	\newcommand{\pd}{\mathrm{PD}}  
	\newcommand{\pdwsx}{\pd_d(\wsx)}  
	\newcommand{\cech}[1][\alpha]{\mathrm{\text{\v{C}ech}}(\pc; #1)}  
	\newcommand{\vr}[1][\alpha]{\mathrm{VR}(\pc; #1)}  

\title{A Higher-Dimensional Homologically Persistent Skeleton}
\author{
	Sara Kali\v{s}nik Verov\v{s}ek\thanks{Max Planck Institute for Mathematics in the Sciences, sara.kalisnik@mis.mpg.de},
	Vitaliy Kurlin\thanks{Materials Innovation Factory, University of Liverpool, vkurlin@liverpool.ac.uk},
	Davorin Le\v{s}nik\thanks{Department of Mathematics, University of Ljubljana, davorin.lesnik@fmf.uni-lj.si
		(this author was partially supported by the Air Force Office of Scientific Research, Air Force Materiel Command, USAF under Award No.~FA9550-14-1-0096)}
}
\date{}

\begin{document}
	
	\maketitle
	\vspace{-1cm}
	
	\begin{abstract}
		Real data is often given as a point cloud, \ie a finite set of points with pairwise distances between them. An important problem is to detect the topological shape of data --- for example, to approximate a point cloud by a low-dimensional non-linear subspace such as an embedded graph or a simplicial complex. Classical clustering methods and principal component analysis work well when given data points split into well-separated clusters or lie near linear subspaces of a Euclidean space.
\bigskip
		
		Methods from topological data analysis in general metric spaces  detect more complicated patterns such as holes and voids that persist for a long time in a 1-parameter family of shapes associated to a cloud. These features can be visualized in the form of a $1$-dimensional homologically persistent skeleton, which optimally extends a minimal spanning tree of a point cloud to a graph with cycles. We generalize this skeleton to higher dimensions and prove its optimality among all complexes that preserve topological features of data at any scale. 
	\end{abstract}
	
	\pagebreak
	\tableofcontents
	\pagebreak
	\section{Introduction}\label{SECTION: Introduction}
	
	\subsection{Motivations and Data Skeletonization Problem}\label{sub:motivations}
		
		Real data is often unstructured and comes in the form of a non-empty finite metric space, called a \df{point cloud}. Such a point cloud can consist of points in 2D images or of high-dimensional vector descriptors of a molecule. A typical problem is to study interesting groups or clusters within data sets.
		
		Real data rarely splits into well-separated clusters, though it often has an intrinsic low-dimensional structure. For example, the point cloud of mean-centered and normalized $3 \times 3$ patches in natural grayscale images has its 50\% densest points distributed near a $2$-dimensional Klein bottle in a $7$-dimensional space~\cite{naturalimages}. This example motivates the following problem.
		
		\smallskip\noindent
		\textbf{Data Skeletonization Problem}.
		Given a point cloud $\pc$ in a metric space $M$, find a complex $\scx \subseteq M$ of a minimal weight to approximate $\pc$ geometrically and topologically in a way that the inclusions of certain subcomplexes of $\scx$ into \df{offsets} of $\pc$ (unions of balls with a fixed radius $\alpha$ and centers at points of $\pc$) induce homology isomorphisms up to a given dimension for all $\alpha$.
		\smallskip
		
		The above problem is harder than describing the topological shape of a point cloud. Indeed, for a noisy random sample $\pc$ of a circle, we aim not only to detect a circular shape $\pc$, but also to approximate an unknown circle by a $1$-dimensional graph $\scx$ that should have exactly one cycle and be close to $\pc$.
		
		To tackle the $1$-dimensional case, Vitaliy Kurlin~\cite{K15} introduced a homologically persistent skeleton (HoPeS) whose cycles are in 1-1 correspondence with all $1$-dimensional persistent homology classes of a given data. The current paper extends the construction and optimality of HoPeS to higher dimensions.
		
		\subsection{Review of Closely Related Past Work}\label{sub:review}
			
			A metric graph reconstruction is related to the data skeletonization problem above. The output is an abstract metric graph or a higher-dimensional complex, which should be topologically similar to an input point cloud $\pc$, but not embedded into the same space as $\pc$, which makes the problem easier.
			
			The classical Reeb graph is such an abstract graph defined for a function $f\colon \cx \to \RR$, where $\cx$ is a simplicial complex built on the points of a given point cloud $\pc$. For example, $\cx$ can be the Vietoris-Rips complex $\vr$ whose simplices are spanned by any set of points whose pairwise distances are at most $2\alpha$. Using the Vietoris-Rips complex at a fixed scale parameter, X.~Ge~et~al.~\cite{GSBW11} proved that under certain conditions the Reeb graph has the expected homotopy type. Their experiments on real data concluded that `there may be spurious loops in the Reeb graph no matter how we choose the parameter to decide the scale'~\cite[Section~3.3]{GSBW11}.
			
			F.~Chazal~et~al.~\cite{CHS15} defined a new abstract $\alpha$-Reeb graph $G$ of a metric space $X$ at a user-defined scale $\alpha$. If $X$ is $\epsilon$-close to an unknown graph with edges of length at least $8\epsilon$, the output $G$ is $34(\beta_1(G)+1)\epsilon$-close to the input $X$, where $\beta_1(G)$ is the first Betti number of $G$~\cite[Theorem~3.10]{CHS15}. The similarity between metric spaces was measured by the Gromov-Hausdorff distance. The algorithm runs at $O(n\log n)$ for $n$ points in $X$.
			
			Another classical approach is to use Forman's discrete Morse theory for a cell complex with a discrete gradient field when one builds a smaller homotopy equivalent complex whose number of critical cells is minimized by the algorithm in~\cite{LLT04}. T.~Dey~et~al.~\cite{DFW13} built a higher-dimensional Graph Induced Complex GIC depending on a scale $\alpha$ and a user-defined graph that spans a cloud $\pc$. If $\pc$ is an $\epsilon$-sample of a good manifold, GIC has the same homology $H_1$ as the Vietoris-Rips complex on $\pc$ at scales $\alpha \geq 4\epsilon$.
			
			A $1$-dimensional homologically persistent skeleton~\cite{K15} is based on a classical minimal spanning tree (MST) of a point cloud. Higher-dimensional MSTs (also called \df{minimal spanning acycles}) are currently a popular topic in the applied topology community, see~\cite{acycle}.
			
			The recent work by P.~Skraba~et~al.~\cite{STY17} studies higher-dimensional MSTs from a probabilistic point of view in the case of \emph{distinctly} weighted complexes, which helps to simplify algorithms and proofs. In practice, simplices often have equal weights, which is a generic non-singular case. For example, in the filtrations of \v{C}ech, Vietoris-Rips and $\alpha$-complexes any obtuse triangle and its longest edge have the same weight equal to the half-length of the longest edge. 
			If arbitrarily small perturbations are allowed to make weights distinct, $\hps$ would become the entire $d$-skeleton of the complex. The more complicated proofs in the paper for non-distinctly weighted complexes are relevant --- it is what makes $\hps$ reasonably small.
			
			Among the important results by P.~Skraba~et~al. is Theorem~3.23~\cite{STY17}, which  establishes a bijection between the set of weights of $d$-simplices outside of a minimal spanning acycle and the set of birth times in the $d$-dimensional persistence diagram. All further constructions and proofs in our paper substantially extend the ideas behind the $1$-dimensional $\hps$ \cite{K15}.

The excellent review by Erickson \cite{E12} discusses an optimization for representatives of homology classes. 
This paper is different in the sense that the weight of a skeleton is globally optimized subject to homological constraints.

We circumvent the NP-hardness result by Chen and Freedman \cite{CF11} for a smallest homology basis by making sure that subcomplexes are spanning.
		
		\subsection{Contributions to Data Skeletonization}\label{sub:contributions}
			
			Definition~\ref{def:HoPeS} introduces a $d$-dimensional homologically persistent skeleton $\hps(\wsx)$ associated to a point cloud $\pc$ or, more generally, to a weighted complex $\wsx$ built on $\pc$. In comparison with the past methods, $\hps(\wsx)$ does not require an extra scale parameter and solves the Data Skeletonization Problem from Subsection~\ref{sub:motivations} in the following sense. 
			
			For any scale parameter $\alpha$, a certain subcomplex of the full skeleton $\hps(\wsx)$ has the minimal total weight among all (in a suitable sense spanning) subcomplexes that have the homology in dimension $d$ of a given weighted complex $\rwsx$ at the same scale $\alpha$ (Theorem~\ref{THEOREM: fittingness and optimality of reduced skeleton}). 
			
			Subskeletons of $\hps$ geometrically approximate the given cloud $\pc$ due to embeddings $\hps_{\alpha}\subseteq \rwsx$ for every scale $\alpha$. These inclusions induce homology isomorphisms, which
justifies a topologically approximation.
			
			The key ingredient of $\hps(\wsx)$ is a $d$-dimensional minimal spanning tree whose new optimality properties are proved in Theorem~\ref{THEOREM: Optimality of Minimal Spanning Trees}. 
Kalai \cite{kalai} introduced similar spanning trees and proved great enumeration results about these complexes, which now found applications in data science.
		
			The construction of $\hps$ for $d=1$ in \cite{K15} did not explicitly define the death times of critical edges when they have equal weights.
			 Example~\ref{Example:Kurlin} shows that extra care is needed when assigning death times in those cases, which has led to new and complete Definition~\ref{def:deaths_critical_faces}.
The implementation for $\alpha$-complexes in the plane \cite{K15} produced correct outputs due to a duality between persistence in dimensions $0$ and $1$. 			
			For completeness, we give a step-by-step algorithm for minimal spanning trees in high dimensions (Algorithm~\ref{ALGORITHM: Minimal spanning tree}) similar to Kruskal~\cite{kruskal} and P.~Skraba~et~al.~\cite[Algorithm~1]{STY17}.

	\section{Preliminaries}\label{SECTION: Preliminaries}
		
		In this section we briefly go over some basic notions and prove basic statements that we will use later in the paper. We start by settling the notation.
		
		\subsection{Notation and the Euler Characteristic}\label{sub:notation}
			
			\begin{itemize}
				\item
					Number sets are denoted by $\NN$ (natural numbers), $\ZZ$ (integers), $\QQ$ (rationals), and $\RR$ (reals). We treat zero as a natural number (so $\NN = \set{0, 1, 2, 3,\ldots}$). We denote the set of extended real numbers by $\eR = \set{-\infty} \cup \RR \cup \set{\infty}$.
				\item
					Subsets of number sets, obtained by comparison with a certain number, are denoted by the suitable order sign and that number in the index. For example, $\NN_{< 42}$ denotes the set $\set{n \in \NN}{n < 42} = \set{0, 1, \ldots, 41}$ of all natural numbers smaller than $42$, and $\RR_{\geq 0}$ denotes the set $\set{x \in \RR}{x \geq 0}$ of non-negative real numbers.
				\item
					Intervals between two numbers are denoted by these two numbers in brackets and in the index. Round, or open, brackets $(~)$ denote the absence of the boundary in the set, and square, or closed, brackets $[~]$ its presence; for example $\intco[\NN]{5}{10} = \set{n \in \NN}{5 \leq n < 10} = \set{5, 6, 7, 8, 9}$.
				\item
					In this paper we work exclusively with finite simplicial complexes. That is, whenever we refer to a `complex' (or a `subcomplex'), we mean a finite simplicial one. By a `$k$-complex' (or a `$k$-subcomplex') we mean a complex of dimension $k$ or smaller. If $\cx$ is a complex, we denote its $k$-skeleton by $\skl[k]{\cx}$.
					
					Formally, we represent any (sub)complex as the set of its simplices (`faces') and any face as the set of its vertices. We will not need orientation for the results in this paper, so this suffices; had we wanted to take orientation into account, we would represent a face as a tuple.
					
					Example of this usage: suppose $\cx$ is a complex, $\scx \subseteq \cx$ and $F \in \cx$. This means that $\scx$ is a subcomplex of $\cx$, $F$ is a face of $\cx$ and $\scx \cup \set{F}$ is the subcomplex of $\cx$, obtained by adding the face $F$ to the subcomplex $\scx$.
				\item
					When we want to refer to the number of $k$-dimensional faces of a complex $\cx$ in a formula, we write $(\# \text{$k$-faces in $\cx$})$.
				\item
					For complexes $\scx \subseteq \cx$ we use $\scx \hookrightarrow \cx$ for the inclusion map. If we have further subcomplexes $\scx' \subseteq \scx$, $\scx' \subseteq \scx'' \subseteq \cx$, we use $(\scx, \scx') \hookrightarrow (\cx, \scx'')$ to denote the inclusion of a pair.
				\item
					\begin{samepage}
					Given $k \in \ZZ$, a unital commutative ring $R$ and a complex $\cx$,
					\begin{itemize}
						\item
							$\ch[k]{\cx}{R}$ stands for the $R$-module of simplicial $k$-chains with coefficients in $R$,
						\item
							$\cy[k]{\cx}{R}$ stands for the submodule of $k$-cycles,
						\item
							$\bd[k]{\cx}{R}$ stands for the submodule of $k$-boundaries,
						\item
							$\hm[k]{\cx}{R}$ stands for the simplicial $k$-homology of $\cx$ with coefficients $R$.
					\end{itemize}
					\end{samepage}
					
					It is convenient to allow the dimension $k$ to be any integer, since we sometimes subtract from it (also, the definition of the $0$-homology does not have to be treated as a special case). Of course, there are no faces of negative dimension, so $\ch[k]{\cx}{R}$, $\cy[k]{\cx}{R}$, $\bd[k]{\cx}{R}$ and $\hm[k]{\cx}{R}$ are all trivial modules whenever $k < 0$.
					
					The boundary maps between chains are denoted by
					\[\partial_k\colon \ch[k]{\cx}{R} \to \ch[k-1]{\cx}{R}.\]
					Given a subcomplex $\scx \subseteq \cx$, these induce boundary maps, defined on the relative homology,
					\[\partial_k\colon \hm[k]{\cx, \scx}{R} \to \hm[k-1]{\scx}{R}.\]
					
					Unless otherwise stated all homologies that we consider in this paper are assumed to be over a given field $\field$, \ie $\hm[k]{\cx}$ stands for $\hm[k]{\cx}{\field}$. Hence $\hm[k]{\cx}$ is a vector space for any $k \in \NN$ and any complex $\cx$; in particular it is free (posseses a basis) and has a well-defined dimension. Since we only consider cases when $\cx$ is a finite complex, the dimension $\beta_k(\cx) \dfeq \dim \hm[k]{\cx}$ (the $k$-th \df{Betti number} of $\cx$) is a natural number, and there exists an isomorphism $\hm[k]{\cx} \ism \field^{\beta_k(\cx)}$.
					
					We freely use the fact that homology is a functor. For a map $f\colon \cx' \to \cx''$ we use $\hm[k]{f}$ to denote the induced map $\hm[k]{\cx'} \to \hm[k]{\cx''}$. (It is common in literature to use the notation $f_*$ for this purpose, but we find it useful to include the dimension in the notation.)
			\end{itemize}
			
			We recall a couple of classical results in topology.
			
			\begin{proposition}\label{PROPOSITION: Euler characteristic}\cite[Chapter~4, Section~3, Corollary~15]{spanier}
				Let $\cx$ be a finite simplicial complex. The alternating sums
				\[\sum_{k \in \NN} (-1)^k (\# \text{$k$-faces in $\cx$}) \qquad \text{and} \qquad \sum_{k \in \NN} (-1)^k \beta_k(\mathcal{Q}) \qquad\qquad\]
				are well defined (all terms with $k > \dim{\cx}$ are zero, so they are effectively finite sums) and equal, regardless of the choice of the field $\field$. The number they are equal to is the \df{Euler characteristic} of $\cx$, and is denoted by $\ec(\cx)$.
			\end{proposition}
			
			\begin{corollary}\label{COROLLARY: adding a face}\cite[Section~3]{EdelsbrunnerFaces}
				Let $\cx$ be a finite simplicial complex, $\scx$ a subcomplex, $k \in \NN$ and $F$ a $k$-face in $\cx$ which is not in $\scx$. Then either
				\begin{itemize}
					\item
						$\beta_{k-1}(\scx \cup \set{F}) = \beta_{k-1}(\scx) - 1$ (``$F$ kills a dimension in $H_{k-1}$'') or
					\item
						$\beta_k(\scx \cup \set{F}) = \beta_k(\scx) + 1$ (``$F$ adds a dimension to $H_k$''),
				\end{itemize}
				while in each case all other Betti numbers are the same for $\scx$ and $\scx \cup \set{F}$.
			\end{corollary}
			
		\subsection{Fitting and Spanning Trees and Forests}\label{sub:subcomplexes}
			
			In order to generalize a 1-dimensional homologically persistent skeleton based on a Minimal Spanning Tree to an arbitrary dimension, we need higher-dimensional analogues of spanning forests and trees. We also define the notion of `fittingness' of a subcomplex.
			
			\begin{definition}\label{DEFINITION: Types of subcomplexes}
				Let $k \in \NN$. Let $\cx$ be a simplicial complex and $\scx$ a $k$-subcomplex of $\cx$.
				\begin{itemize}
					\item
						$\scx$ is \df{$k$-spanning} (in $\cx$) when $\skl[k-1]{\scx} = \skl[k-1]{\cx}$, \ie the $(k-1)$-skeleton of $\scx$ is the entire $(k-1)$-skeleton of $\cx$.
					\item
						$\scx$ is a \df{$k$-forest} (in $\cx$) when $\hm[k]{\scx} = 0$.
					\item
						$\scx$ is a \df{$k$-tree} (in $\cx$) when it is a $k$-forest and $\hm[k-1]{\scx \to \bullet}$ is an isomorphism.\footnote{Here $\bullet$ denotes a singleton, so there is a unique map $\scx \to \bullet$. If $k \neq 1$, the condition for $\scx$ being a $k$-tree simplifies to $\hm[k]{\scx} = \hm[k-1]{\scx} = 0$. For $k = 1$, the induced map $\hm[k-1]{\scx \to \bullet}$ is an isomorphism if and only if $\scx$ has exactly one connected component.}
					\item
						$\scx$ is \df{$k$-fitting} (in $\cx$) when $\hm[i]{\scx \hookrightarrow \cx}$ is an isomorphism for all $i \in \NN_{\leq k}$.
				\end{itemize}
				For the sake of simplicity, we shorten `$k$-spanning $k$-forest' to a `spanning $k$-forest' (or to `spanning forest', when $k$ is understood). We proceed similarly with trees.
			\end{definition}
			
			Note that every subcomplex, including $\emptyset$, is $0$-spanning, since the $(-1)$-skeleton is empty. Also, $\emptyset$ is the only $0$-forest and the only $0$-tree.
			
			\begin{example}
				Let $T$ be the set of all non-empty subsets of a set with four elements, \ie a geometric realization of $T$ is a tetrahedron. Then $T$ is a spanning $3$-tree of itself. Figure~\ref{FIGURE: Tetrahedron} depicts two spanning $2$-trees of $T$.
				\begin{figure}[!ht]
					\centering
					\includegraphics[scale=1]{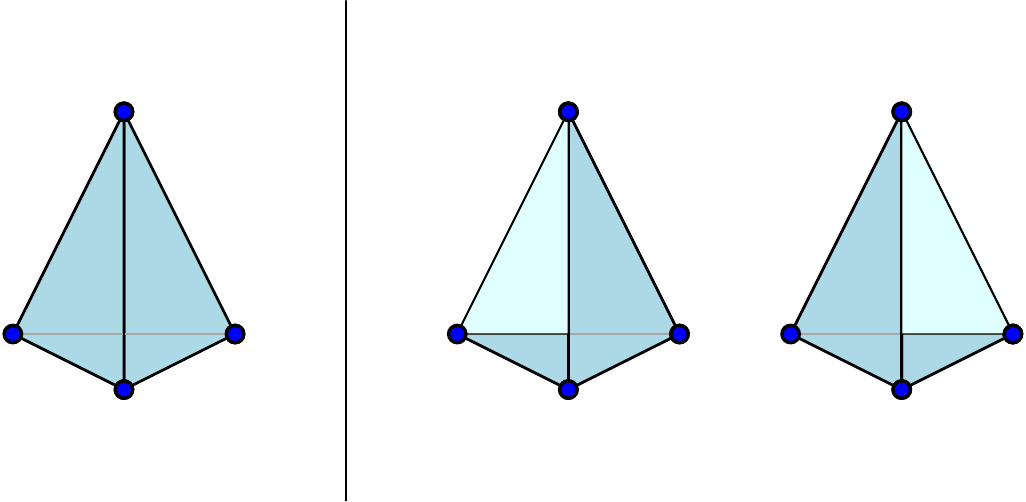}
					\caption{Geometric realization of a tetrahedron $T$ (left) and two of its spanning $2$-trees (right).}\label{FIGURE: Tetrahedron}
				\end{figure}
			\end{example}
			
			\begin{remark}
				The concepts in Definition~\ref{DEFINITION: Types of subcomplexes} were inspired by~\cite{10.2307/1997562, 2015arXiv150606819D}.
				However, the data skeletonization problem in subsection~\ref{sub:motivations} a different version of a high-dimensional tree.
In particular, the definition of a $k$-forest in~\cite{2015arXiv150606819D} was given in an `absolute' sense, as linear independence of the columns of the boundary map $\partial_k$ between $\ZZ$-chains. This is equivalent to $\hm[k]{\scx}{\RR} = 0$ (or more generally, $\hm[k]{\scx}{\field} = 0$ if $\field$ is a field of characteristic $0$). However, we purposefully define forests (and trees) in a `relative' sense (depending on the choice of the field $\field$), as this allows us to prove the results of the paper in greater generality.
			\end{remark}
			
			\begin{remark}
				What we call a spanning $k$-tree some other authors~\cite{acycle, STY17} call a $k$-spanning acycle. This definition originated in Kalai's work~\cite{kalai}. He considered $k$-dimensional simplicial complexes, which contain the entire $(k-1)$-skeleton and for them defined `simplicial spanning trees'.
			\end{remark}
			
			The following lemma establishes basic properties of spanning subcomplexes that we use throughout the paper.
			
			\begin{lemma}\label{LEMMA: spanning subcomplex}
				Let $\cx$ be a finite simplicial complex and $\scx$ a $k$-spanning $k$-subcomplex of $\cx$ for some $k \in \NN$.
				\begin{enumerate}
					\item\label{LEMMA: spanning subcomplex: lower-dimensional fittingness}
						The map $\hm[i]{\scx \hookrightarrow \cx}$ is an isomorphism for all $i \in \NN_{\leq k-2}$ (\ie $\scx$ is $(k-2)$-fitting in $\cx$) and a surjection for $i = k-1$.
					\item\label{LEMMA: spanning subcomplex: formula}
						The formula
						\[
							\Big(\# \text{$k$-faces in $\cx$}\Big) + \beta_{k-1}(\cx^{(k)}) - \beta_k(\cx^{(k)}) \sepeq \Big(\# \text{$k$-faces in $\scx$}\Big) + \beta_{k-1}(\scx) - \beta_k(\scx)
						\]
						holds.
					\item\label{LEMMA: spanning subcomplex: towards fittingness}
						If $\beta_{k-1}(\scx) > \beta_{k-1}(\cx)$, there exists a $k$-face $F$ in $\cx \setminus \scx$ such that
						\[
							\beta_k(\scx \cup \set{F}) = \beta_k(\scx) \quad \text{and} \quad \beta_{k-1}(\scx \cup \set{F}) = \beta_{k-1}(\scx) - 1.
						\]
					\item\label{LEMMA: spanning subcomplex: existence of a nice subforest}
						If $\scx$ is $(k-1)$-fitting in $\cx$, a $k$-subcomplex $F \subseteq \scx$ exists, which is $(k-1)$-fitting $k$-spanning $k$-forest in $\scx$ (and consequently also in $\cx$).
					\item\label{LEMMA: spanning subcomplex: connecting relative and absolute homology}
						Suppose $\scx$ is $(k-1)$-fitting in $\cx$ and $F \subseteq \scx$ is a $(k-1)$-fitting $k$-spanning $k$-forest in $\scx$ (equivalently, in $\cx$). Then the diagram
						\[
							\begin{tikzpicture}[baseline=(current  bounding  box.center), scale = 1]
								\node (Ta) at (5, 2) {$ \hm[k]{\scx}$};
								\node (Tb) at (5,0) {$\hm[k]{\cx}$};
								\node (Tc) at (8.5, 2) {$\hm[k]{\scx, F}$};
								\node (Td) at (8.5, 0) {$\hm[k]{\cx, F}$};
								\path[->]
								(Ta) edge node[left] {{\small $\hm[k]{\scx \hookrightarrow \cx}$}} (Tb)
								(Tc) edge node[right] {{\small $\hm[k]{(\scx, F) \hookrightarrow (\cx, F)}$}} (Td)
								(Ta) edge node[above, yshift=0.5cm] {{\small $\hm[k]{(\scx, \emptyset) \hookrightarrow (\scx, F)}$}} (Tc)
								(Tb) edge node[below, yshift=-0.5cm] {{\small $\hm[k]{(\cx, \emptyset) \hookrightarrow (\cx, F)}$}} (Td);
							\end{tikzpicture}
						\]
						commutes and the horizontal arrows are isomorphisms. Hence the left arrow is an isomorphism if and only if the right one is.
				\end{enumerate}
			\end{lemma}
			
			\begin{proof}
				\begin{enumerate}
					\item
						Follows from the fact that simplicial homology in dimension $i$ depends only on $i$-~and $(i+1)$-dimensional faces, with $i$-faces providing the generators and $(i+1)$-faces the relations.
					\item
						Since $\scx$ is $k$-spanning, it has the same number of faces up to dimension $k-1$ and (per the previous item) the same homologies up to dimension $k-2$. Thus
						\[
							(-1)^k \big(\# \text{$k$-faces in $\cx$} - \# \text{$k$-faces in $\scx$}\big) = \ec(\cx^{(k)}) - \ec(\scx) =
						\]
						\[
							= (-1)^k \beta_k(\cx^{(k)}) + (-1)^{k-1} \beta_{k-1}(\cx^{(k)}) - (-1)^k \beta_k(\scx) - (-1)^{k-1} \beta_{k-1}(\scx).
						\]
						After rearranging the result follows.
					\item
						Let $\set{S_1, S_2, \ldots, S_m}$ be the set of $k$-faces in $\scx$. Consider families of {$k$-faces} in $\cx \setminus \scx$ which, when added to $\scx$, reduce the $(k-1)$-homology (regardless of whether the $k$-th Betti number of the expanded subcomplex changes). By assumption ${\beta_{k-1}(\scx) > \beta_{k-1}(\cx)}$, at least one such family exists, namely the set of \emph{all} $k$-faces in $\cx \setminus \scx$. Let $\mathcal{F} = \set{F_1, F_2, \ldots, F_n}$ be one of such families which contains the minimal possible number of $k$-faces (of course $n \geq 1$). Minimality of $\mathcal{F}$ implies that the image under $\partial_k$ does not change when adding only $n-1$ faces, that is
						\[
							\partial_k\big(\linspan{S_1, S_2, \ldots, S_m}\big) = \partial_k\big(\linspan{S_1, S_2, \ldots, S_m, F_1, F_2, \ldots, F_{n-1}}\big) \dfeqrev B
						\]
						(here $\linspan{~}$ denotes the linear span). Since adding $\mathcal{F}$ to $\scx$ reduces $(k-1)$-homology, a linear combination
						\[
							s \dfeq \sum_{i = 1}^m c_i S_i + \sum_{j = 1}^n d_j F_j
						\]
						exists such that $\partial_k(s) \notin B$. Consequently $\partial_k(F_n) \notin B$, so just adding $F_n$ to $\scx$ reduces homology in dimension $(k-1)$ (implying that $n = 1$). It follows from Corollary~\ref{COROLLARY: adding a face} that $\scx \cup \set{F_n}$ remains a $k$-forest while $\beta_{k-1}(\scx \cup \set{F_n}) = \beta_{k-1}(\scx) - 1$.
					\item
						Let $\set{S_1, S_2, \ldots, S_m}$ be the set of $k$-faces in $\scx$. Since $\scx$ is a $k$-complex, we have $\hm[k]{\scx} \ism \cy[k]{\scx}$ (every equivalence class is a singleton). Let $n \dfeq \beta_k(\scx)$ be the dimension of the vector space of $k$-cycles of $\scx$. Choose a basis $b_1, \ldots, b_n$ of $\cy[k]{\scx}$ and expand these basis elements as
						\[
							b_i = \sum_{j = 1}^m c_{ij} S_j.
						\]
						Consider the system of linear equations
						\[
							\sum_{j = 1}^m c_{ij} x_j = 0.
						\]
						Since a basis is linearly independent, this is a system of $n$ independent linear equations with $m$ variables, where $n \leq m$ (since $\cy[k]{\scx} \subseteq \ch[k]{\scx}$). Thus the system can be solved for $n$ leading variables in the sense that we express them with the remaining $m-n$ ones. Without loss of generality assume that the first $n$ variables are the leading ones. This means that the system can be equivalently written as
						\[
							x_i + \sum_{j = n+1}^m \widetilde{c}_{ij} x_j = 0.
						\]
						Define $\displaystyle{\widetilde{b}_i \dfeq S_i + \sum_{j = n+1}^m \widetilde{c}_{ij} S_j}$; then $\set{\widetilde{b}_i}{i \in \intcc[\NN]{1}{n}}$ is also a basis for $\cy[k]{\scx}$.
						
						Define $F \dfeq \scx \setminus \set{S_i}{i \in \intcc[\NN]{1}{n}}$. Clearly $F$ is $k$-spanning (therefore $(k-2)$-fitting) in $\scx$ and $\cx$. Let $z = \sum_{j = n+1}^m d_j S_j$ be an arbitrary $k$-cycle of $F$. The boundary map has the same definition for $F$ and $\scx$, so $z$ is also a cycle in $\scx$. Expand it as
						\[
							z = \sum_{i = 1}^n e_i \widetilde{b}_i.
						\]
						Since $z$ does not include any $S_j$ for $j \leq n$, necessarily all $e_i$s are zero, and then $z = 0$. We conclude that $F$ is a $k$-forest.
						
						Adding $n$ faces to $F$ to recover $\scx$ increases the dimension of $k$-homology by $n$. Since a change of a $k$-face either modifies the dimension of $k$-homology by one or of $(k-1)$-homology by one (Corollary~\ref{COROLLARY: adding a face}), the $(k-1)$-homology of $F$ remains the same as of $\scx$. That is, $F$ is $(k-1)$-fitting in $\scx$ and $\cx$.
					\item
						The long exact sequence of a pair is natural, so the following diagram commutes.
						\[
							\begin{tikzpicture}[baseline=(current  bounding  box.center), scale = 0.8]
								\node (homa) at (2, 2) {$\overbrace{\hm[k]{F}}^{0}$};
								\node (homb) at (2,0) {$\underbrace{\hm[k]{F}}_{0}$};
								\node (Ta) at (5, 2) {$ \hm[k]{\scx}$};
								\node (Tb) at (5,0) {$\hm[k]{\cx}$};
								\node (Tc) at (8.5, 2) {$\hm[k]{\scx, F}$};
								\node (Td) at (8.5, 0) {$\hm[k]{\cx, F}$};
								\node (Te) at (11.8, 2) {$\hm[k-1]{F}$};
								\node (Tf) at (11.8, 0) {$\hm[k-1]{F}$};
								\node (Tg) at (15.3, 2) {$\hm[k-1]{\scx}$};
								\node (Th) at (15.3, 0) {$\hm[k-1]{\cx}$};
								\path[->]
								(homa) edge node[left]{} (homb)
								(Ta) edge node[right] {} (Tb)
								(Tc) edge node[right] {} (Td)
								(homa) edge node[above] {0} (Ta)
								(homb) edge node[below] {0} (Tb)
								(Ta) edge node[above] {$\ism$} (Tc)
								(Tb) edge node[below] {$\ism$} (Td)
								(Te) edge node[below] {} (Tf)
								(Tc) edge node[above] {0} (Te)
								(Td) edge node[below] {0} (Tf)
								(Te) edge node[above] {$\ism$} (Tg)
								(Tf) edge node[below] {$\ism$} (Th)
								(Tg) edge node[below] {} (Th);
							\end{tikzpicture}
						\]
						Since $\hm[k]{F} = 0$, the outgoing maps are $0$. Since $F$ is $(k-1)$-fitting, the maps $\hm[k-1]{F \hookrightarrow \scx}$ and $\hm[k-1]{F \hookrightarrow \cx}$ are isomorphisms, so the preceding boundary maps are $0$. Thus the maps $\hm[k]{(\scx, \emptyset) \hookrightarrow (\scx, F)}$ and $\hm[k]{(\cx, \emptyset) \hookrightarrow (\cx, F)}$ are isomorphisms.
				\end{enumerate}
			\end{proof}
			
			\begin{proposition}\label{PROPOSITION: spanning trees}
				Let $k, n \in \NN$ and let $\Delta_n$ be a standard $n$-simplex. The following statements hold.
				\begin{enumerate}
					\item\label{PROPOSITION: spanning trees: existence}
						There exists a spanning $k$-tree in $\Delta_n$.
					\item\label{PROPOSITION: spanning trees: counting faces}
						The number of $k$-faces in any spanning $k$-tree in $\Delta_n$ is $\binom{n}{k}$ if $k \geq 1$, and $0$ if $k = 0$.
					\item\label{PROPOSITION: spanning trees: maximality}
						Let $F$ be a spanning $k$-forest in $\Delta_n$. Then $F$ is a $k$-tree if and only if it is a maximal $k$-forest in the sense that for every $k$-face $E \in \Delta_n \setminus F$ we have ${\hm[k]{F \cup \set{E}} \neq 0}$.
				\end{enumerate}
			\end{proposition}
			
			\begin{proof}
				\begin{enumerate}
					\item
						This follows if we apply Lemma~\ref{LEMMA: spanning subcomplex}(\ref{LEMMA: spanning subcomplex: existence of a nice subforest}) for $\skl[k]{\Delta_n} \subseteq \Delta_n$, but we can be much more explicit.
						
						If $k = 0$, then $\emptyset$ is a spanning $0$-tree. If $k \geq 1$, choose a vertex $v$ in $\Delta_n$. Define $T$ to consist of the $(k-1)$-skeleton of $\Delta_n$, as well as of those $k$-faces of $\scx$ which contain $v$. Then $T$ is $k$-spanning by definition, and there exists an obvious deformation retraction of $T$ onto $v$. This deformation retraction induces homology isomorphisms in all dimensions, so $T$ is necessarily a tree.
					\item
						The only spanning $0$-tree is $\emptyset$, so the statement holds for $k = 0$. Assume $k \geq 1$. Let $T$ be any spanning $k$-tree in $\Delta_n$ and let $x$ be the number of $k$-faces of $T$. Counting the number of faces, we obtain
						\[
							\ec(T) = \Big(\sum_{i \in \NN_{\leq k-1}} (-1)^i \binom{n+1}{i+1}\Big) + (-1)^k x = -\Big((-1)^k \binom{n}{k} - 1\Big) + (-1)^k x.
						\]
						On the other hand, since $T$ is $k$-spanning, it has the same homology up to dimension $k-2$ as the standard simplex $\Delta_n$, and thus the same homology up to dimension $k-2$ as a point. Since $T$ is a $k$-tree, this holds also for the dimensions $k-1$ and $k$. Hence
						\[
							\ec(T) = \sum_{i \in \NN_{\leq k}} \beta_i(T) = 1.
						\]
						Equating the two versions of the Euler characteristic (as in Proposition~\ref{PROPOSITION: Euler characteristic}), we obtain $x = \binom{n}{k}$.
					\item
						Clearly the statement holds for the only $0$-forest $F = \emptyset$. Assume hereafter that $k \geq 1$.
						\begin{description}
							\item{$(\impl)$}\\
								Suppose $F$ is a $k$-tree. By Corollary~\ref{COROLLARY: adding a face}, adding $E$ to $F$ either decreases $\beta_{k-1}$ by $1$ or increases $\beta_k$ by $1$. The former is impossible: if $k \geq 2$, then $\hm[k-1]{F}$ is already trivial, and if $k = 1$ (therefore $\beta_{k-1}(F) = 1$), adding a face cannot decrease the number of connected components to zero.
								
								Hence $\beta_k(F \cup \set{E}) = 1$, so $F \cup \set{E}$ is not a $k$-forest.
							\item{$(\revimpl)$}\\
								Apply basic graph theory if $k = 1$ ($1$-forests and $1$-trees are just the usual forests and trees). Suppose $k \geq 2$ and assume that the spanning $k$-forest $F$ is not a $k$-tree, so $\beta_{k-1}(F) > 0 = \beta_{k-1}(\Delta_n)$. Use Lemma~\ref{LEMMA: spanning subcomplex}(\ref{LEMMA: spanning subcomplex: towards fittingness}) to find a $k$-face $E \in \Delta_n \setminus F$ with $\beta_k(F \cup \set{E}) = \beta_k(F) = 0$, contradicting the assumption.
						\end{description}
				\end{enumerate}
			\end{proof}
			
		\subsection{Filtrations on a Point Cloud}
			
			In practice, point clouds are often obtained by sampling from a particular shape, which we want to reconstruct. However, from the point of view of a topologist, point clouds themselves do not have an interesting shape --- the dimension of $0$-homology is the number of points in the point cloud and the higher-dimensional homology groups are all trivial. The idea is to assume that the point cloud is a subspace of a larger metric space (let us denote its metric by $D$), typically some Euclidean space $\RR^N$, in which each point can be thickened to a ball of some specified radius $\alpha$. The union of these balls is called the \df{$\alpha$-offset} of $\pc$ and is denoted by $\pc(\alpha)$, see Figure~\ref{FIGURE: offsets}.
			
			\begin{figure}[!ht]
				\centering
				\hspace*{-1.5cm}
				\begin{tabular}{ccc}
					\includegraphics[scale=0.63]{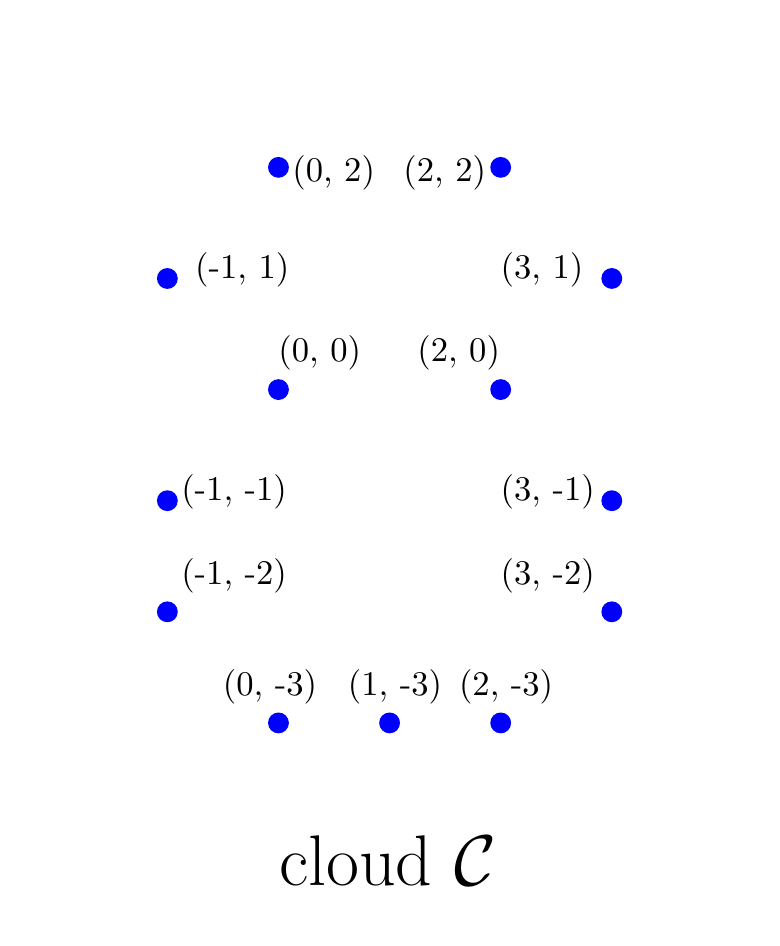} &
					\includegraphics[scale=0.63]{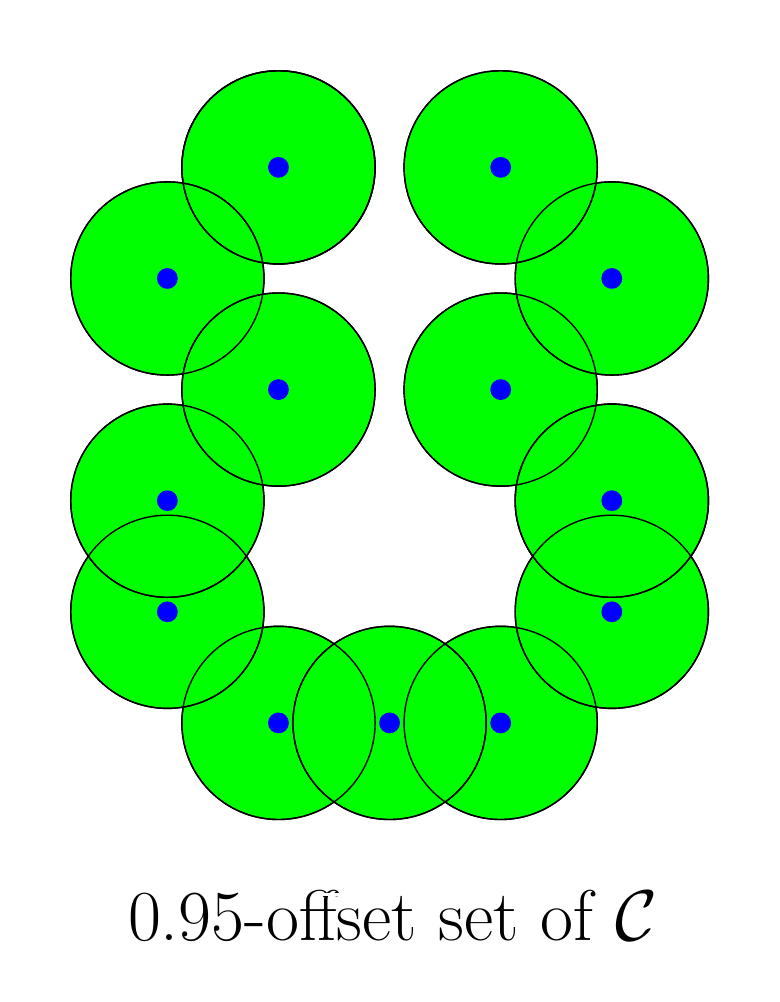} &
					\includegraphics[scale=0.63]{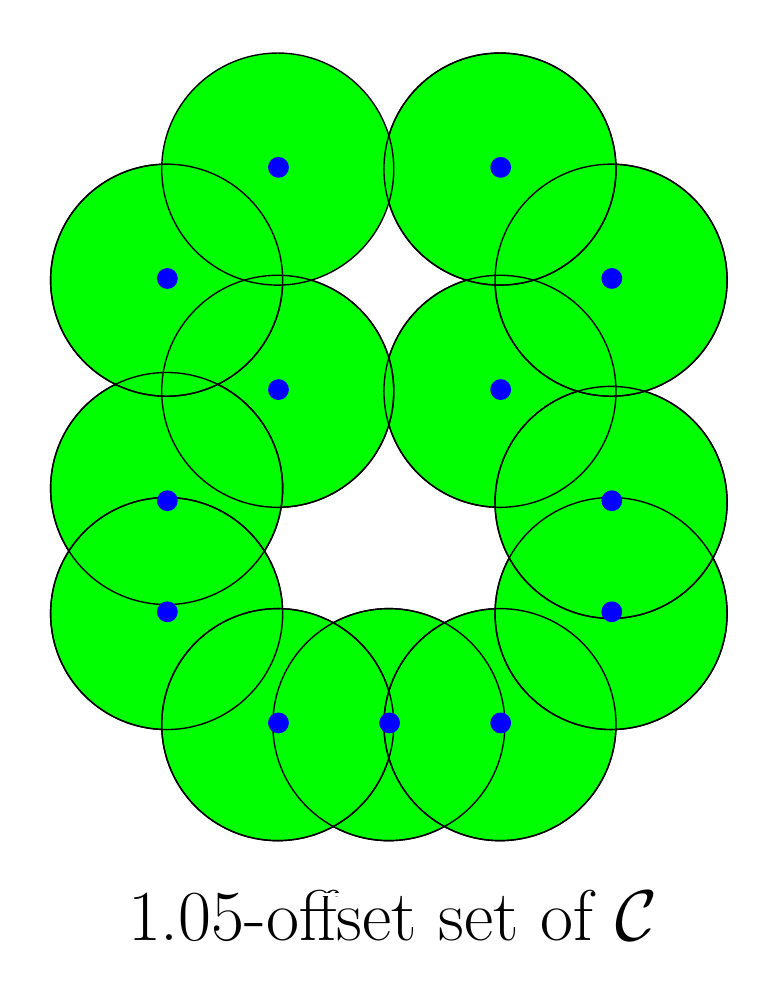}
				\end{tabular}
				\caption{Point cloud $\pc$ and two example offsets of $\pc$. The $1.05$-offset has non-trivial first homology.}\label{FIGURE: offsets}
			\end{figure}
			
			The nerve of $\pc(\alpha)$ is called the \df{\v{C}ech complex} $\cech$ of $\pc$ at $\alpha$. The nerve lemma~\cite{Bjorner:1996:TM:233228.233246} says that the homotopy type of $\cech$ is the same as the homotopy type of $\pc(\alpha)$. Hence, $\cech$ is a potentially good approximation to the shape, from which we sampled the point cloud.
			
			For any $\alpha < \alpha'$, we have the inclusion $\cech[\alpha] \subseteq \cech[\alpha']$. That is, the collection $\big(\cech\big)_{\alpha \in \RR}$ is a \df{filtration}.
			
			\v{C}ech filtration is not ideal for computation, as it requires storing all high-dimensional simplices in a computer memory. On the other hand, the filtration of \df{Vietoris-Rips complexes} is completely determined by the $1$-dimensional skeleton. For any scale $\alpha \in \RR$, the complex $\vr$ has a $k$-dimensional simplex on points $v_0, \ldots, v_k \in \pc$ whenever all pairwise distances $D(v_i, v_j) \leq 2 \alpha$ for all $0 \leq i < j \leq k$.
			
		\subsection{Persistent Homology of a Filtration}
			
			For excellent introductions to persistent homology, see~\cite{ghrist, topodata, pattern, elz-tps-02}. The usual homology is defined for a single complex, but the key idea of persistence is to consider an entire filtration of complexes $\filt$, rather that just a single stage $\stage$ at a specific scale parameter $\alpha$. The reason for this is that it is hard (or even impossible) to choose a single parameter value in a way that assures that $\stage$ is a good approximation to the shape we sampled the point cloud from. Also, choosing a single parameter value is highly unstable. 
			
				\begin{figure}[!ht]
				\centering
				\hspace*{-1cm}
				\begin{tabular}{cc}
					\includegraphics[scale=0.63]{cloudcoord.pdf} &
					\includegraphics[scale=1.2]{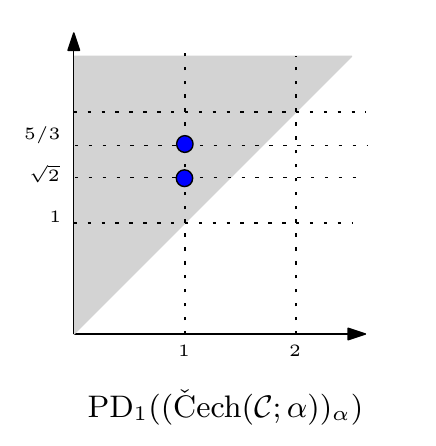}
				\end{tabular}
				\caption{Point cloud $\pc$ and its persistence diagram in dimension~$1$ (with homology coefficients~$\RR$), obtained via a filtration of \v{C}ech complexes on $\pc$. Each point in the diagram represents a cycle present over a range of parameters $\alpha$.}\label{FIGURE: Persistent diagram}
			\end{figure}
			
			Persistent homology in dimension $k$ tracks changes in the $k$-homology $\hm[k]{\stage}$ over a range of scales $\alpha$. This information can be summarized by a  \df{persistence diagram} $\pd_k\Big(\filt\Big)$. A dot $(p, q)$ in a persistence diagram represents an interval $\intco{p}{q}$ corresponding to a topological feature, a $k$-dimensional void, which appears at $p$ and disappears at $q$. These barcodes play the same role as a histogram would in summarizing the shape of data, with long intervals corresponding to strong topological signals and short ones to noise.	
			
			In Figure~\ref{FIGURE: Persistent diagram} the persistence diagram $\pd_1\Big(\big(\cech\big)_{\alpha \in \RR}\Big)$ consists of 2 dots. The dot $(1, \sqrt{2})$ says that a $1$-dimensional cycle enclosing the smaller hole (the upper bounded component of $\RR^2 \setminus \cech$) was born at $\alpha = 1$ and died at $\alpha = \sqrt{2}$ when this hole was filled. Similarly, the dot $(1, \frac{5}{3})$ says that the larger hole persisted from the same birth time $\alpha = 1$ until the later death time $\alpha = \frac{5}{3}$.

		\subsection{Weighted Simplices}
			
			Given a filtration, we can assign to any face in it its \df{weight} as the parameter value when it appears in the filtration. The union of all stages in the filtration, together with the weights of all simplices, is thus a weighted complex (a higher-dimensional analogue of weighted graphs).
			
			For both \v{C}ech and Vietoris-Rips filtrations of a point cloud $\pc$, the simplicial complex for parameter values $\alpha \geq \max_{v_i, v_j \in \pc} \frac{D(v_i, v_j)}{2}$ is a full simplex on $|\pc|$ vertices. Thus we can think of the whole filtration as being encoded by a weighted \emph{simplex}. For this reason most of the results stated in this paper are in terms of weighted simplices. If a certain filtration does not terminate with a full simplex, we can always complete the simplicial complex at the last step to a full simplex by adding the missing faces and assigning them weight bigger than that of all faces in the original filtration.
			
			The main reason to work with a single weighted simplex is to have a simpler notion of a minimal spanning $d$-tree at the last stage of the filtration. Otherwise `minimal spanning $d$-tree' should be replaced with `minimal $(d-1)$-fitting $d$-spanning $d$-forest' and arguments would get more complicated.
			
			We give a formal definition of a weighted simplex. As mentioned in the subsection on notation, we will not need orientation, so we can encode faces with sets, rather than tuples.
			
			\begin{definition}\label{def:weighted}
				Given a set $\pc$, let $\inhpst(\pc)$ denote the set of non-empty subsets of $\pc$. 
				\begin{itemize}
					\item
						A \df{weighting} on a set $\pc$ is a map $w\colon \inhpst(\pc) \to \RR_{\geq 0}$ which is monotone in the sense that if $\emptyset \neq A \subseteq B \subseteq \pc$, then $w(A) \leq w(B)$. For any $A \in \inhpst(\pc)$ the value $w(A)$ is the \df{weight} of $A$ (relative to the weighting $w$).
					\item
						A \df{weighted simplex} is a pair $(\pc, w)$, where $\pc$ is a non-empty finite set and $w$ a weighting on it. We denote $\wsx \dfeq (\pc, w)$ for short.
					\item
						For a weighted simplex $\wsx$ and any family of subsets $\mathscr{S} \subseteq \inhpst(\pc)$ we denote its \df{total weight} by $\tw(\mathscr{S}) \dfeq \sum_{A \in \mathscr{S}} w(A)$.
				\end{itemize}
			\end{definition}
			
			Monotonicity of weighting implies that
			\[
				\rwsx \dfeq \set[1]{A \in \inhpst(\pc)}{w(A) \leq \alpha}
			\]
			is a subcomplex for any $\alpha \in \RR$, and $(\rwsx)_{\alpha \in \RR}$ is a filtration. Note that the image of a weighting is a finite subset of $\intcc{0}{w(\pc)}$, and we have $\rwsx = \inhpst(\pc)$ for all $\alpha \in \RR_{\geq w(\pc)}$.
			
			Conversely, a filtration $\filt$ induces a weighted \emph{complex} with the weighting
			\[
				w(A) = \sup\set{\alpha \in \RR}{A \notin \stage} = \inf\set{\alpha \in \RR}{A \in \stage},
			\]
			and we get a weighted \emph{simplex} whenever each non-empty subset of $\pc$ appears in the filtration at a specific time $\alpha \in \RR_{\geq 0}$.
			
			In the specific case of \v{C}ech filtration, the weighting is given by
			\[
				w(A) \dfeq \inf\set[1]{\alpha \in \RR_{\geq 0}}{\xsome{x}{X}\xall{a}{A}{D(a, x) \leq \alpha}},
			\]
			and in the case of Vietoris-Rips filtration by
			\[
				w(A) \dfeq \frac{1}{2} \cdot \sup_{a, b \in A} D(a, b)
			\]
			for $A \in \inhpst(\pc)$.

	\section{Minimal Spanning $d$-Tree}\label{SECTION: Minimal Spanning Tree}
		
		The first step in constructing a 1-dimensional homologically persistent skeleton in~\cite{K15} was to take a classical ($1$-dimensional) Minimal Spanning Tree of a given point cloud. With this idea in mind, we generalize the concept of a minimal spanning tree to higher dimensions. Hereafter fix a weighted simplex $\wsx$ and a dimension $d \in \NN$.
		
		\begin{definition}[Minimal Spanning Tree]\label{DEFINITION: Minimal spanning tree}
			A \df{minimal spanning $d$-tree} (or simply \df{minimal spanning tree} when $d$ is understood) of $\wsx$ is a spanning $d$-tree of $\wsx$ with minimal total weight. We use $\mst(\wsx)$ to denote any chosen minimal spanning tree, and shorten this to $\mst$ when $\wsx$ is understood from the context. For any $\alpha \in \RR$  we define
			\[
				\rmst(\wsx) \dfeq \set{A \in \mst(\wsx)}{w(A) \leq \alpha}
			\]
			and shorten this to $\rmst$ when there is no ambiguity.
		\end{definition}
		
		By Proposition~\ref{PROPOSITION: spanning trees}(\ref{PROPOSITION: spanning trees: existence}) a spanning $d$-tree of $\wsx$ exists, and so a minimal spanning $d$-tree exists also. In general there may be many minimal spanning trees; for example, any two edges form a minimal spanning $1$-tree in an equilateral triangle.
		
		In the next subsection we give an explicit construction for a minimal spanning tree and then prove that all minimal spanning trees are obtained this way. This allows us to later prove optimality of minimal spanning trees at all scales (Theorem~\ref{THEOREM: Optimality of Minimal Spanning Trees}).
		
		\subsection{Construction of Minimal Spanning $d$-Trees}\label{SUBSECTION: Construction of a Minimal Spanning Tree}
		
			The idea to obtain a minimal spanning tree is to go through the image of $w$ and inductively construct a $(d-1)$-fitting $d$-spanning $d$-forest $\armst$ in $\rwsx$, with minimal total weight among such, for every $\alpha \in \RR$.
			
			Let $w_1 < w_2 < \ldots < w_n$ be all elements of $\im(w)$ and set additionally $w_0 = -\infty$, $w_{n+1} = \infty$. Declare $\armst \dfeq \emptyset$ for all $\intco[\eR]{-\infty}{w_1}$.
			
			Take $k \in \intcc[\NN]{1}{n}$ and assume that $\armst[\gamma]$ has been defined for all $\gamma < w_k$. We define $\armst$ for $\alpha \in \intco{w_k}{w_{k+1}}$ to consist of the subcomplex we had at the previous stage, but with as many faces of weight $w_k$ added as possible while still keeping the subcomplex a forest.
			
			Explicitly, let $F_1, F_2, \ldots, F_m$ be all $d$-faces of weight $w_k$.\footnote{The order of these faces can be chosen arbitrarily. It is because of this freedom that there are in general many minimal spanning trees.} Define $\scx_0$ to be the union of $\armst[w_{k-1}]$ with the set of all faces in $\wsx$ which have the weight $w_k$ and dimension at most $d-1$. Note that $\scx_0$ is a spanning $d$-forest in $\rwsx$.
			
			Suppose inductively that we have defined a spanning $d$-forest $\scx_{i-1}$, where $i \in \intcc[\NN]{1}{m}$. If $\scx_{i-1} \cup \set{F_i}$ is still a $d$-forest, define $\scx_i \dfeq \scx_{i-1} \cup \set{F_i}$, otherwise define $\scx_i \dfeq \scx_{i-1}$. In the end, set $\armst \dfeq \scx_m$ which is a spanning $d$-forest by construction.
			
			Here is the summary of this procedure, written as an explicit algorithm.
			
			\begin{algorithm}[H]
				\caption{Construction of a minimal spanning $d$-tree}\label{ALGORITHM: Minimal spanning tree}
				\begin{algorithmic}[1]
					\STATE $w_0 \dfeq -\infty$
					\STATE $w_1, w_2, \ldots, w_n \dfeq$ elements of $\im(w)$, in order
					\STATE $w_{n+1} \dfeq \infty$
					\STATE $\armst \dfeq \emptyset$ for all $\alpha \in \intco[\eR]{-\infty}{w_1}$
					\FOR{$k = 1$ \TO $n$}
						\STATE $F_1, F_2, \ldots, F_m \dfeq$ $d$-faces of weight $w_k$ in $\wsx$
						\STATE $\scx_0 \dfeq \armst[w_{k-1}] \cup \set[1]{A \in \skl[d-1]{\wsx}}{w(A) = w_k}$
						\FOR{$i = 1$ \TO $m$}
							\IF{$\beta_d(\scx_{i-1} \cup \set{F_i}) = 0$}
								\STATE $\scx_i \dfeq \scx_{i-1} \cup \set{F_i}$
							\ELSE
								\STATE $\scx_i \dfeq \scx_{i-1}$
							\ENDIF
						\ENDFOR
						\STATE $\armst \dfeq \scx_m$ for all $\alpha \in \intco[\eR]{w_k}{w_{k+1}}$
					\ENDFOR
				\end{algorithmic}
			\end{algorithm}
			
			\begin{example}\label{MSTexample}
				Let $\pc$ be a point cloud consisting of four vertices with pairwise distances as specified on the left-hand side of Figure~\ref{MST}. The right-hand side of Figure~\ref{MST} depicts a minimal spanning $2$-tree at different scales $\alpha$. The weighting is induced by the \v{C}ech filtration on $\pc$.
				
				\begin{figure}[!ht]
					\centering
					\includegraphics{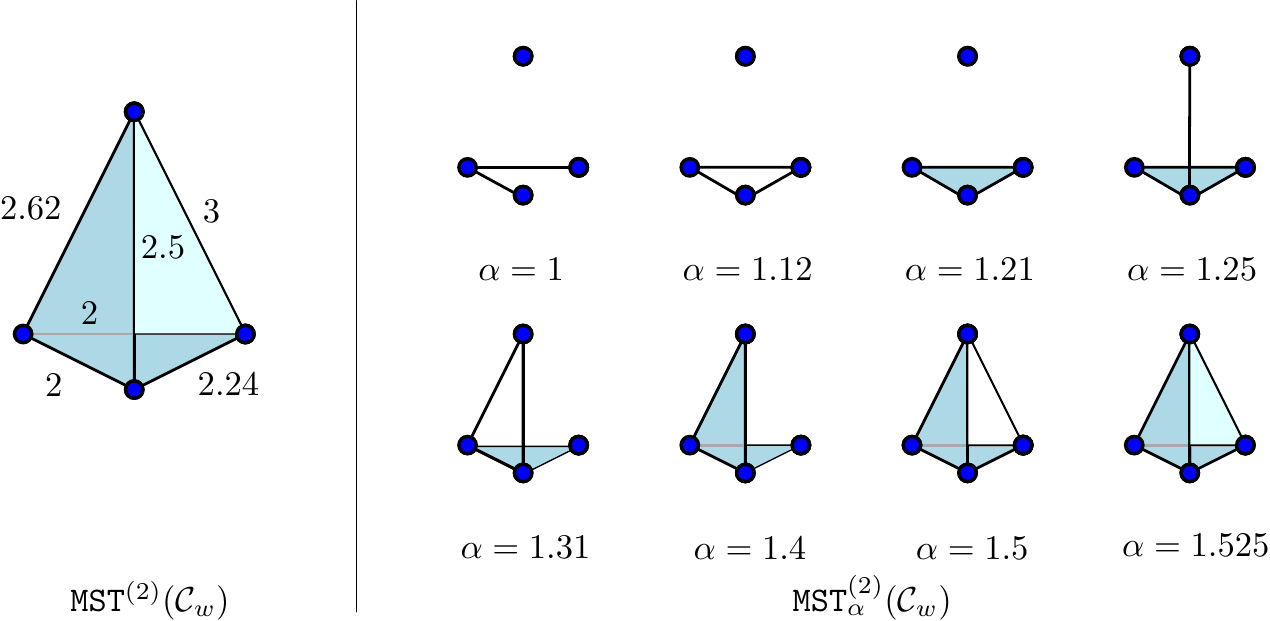}
					\caption{Geometric realizations of $\mathtt{MST}^{(2)}(\wsx)$ and its reduced forms with respect to \v{C}ech filtration of a point cloud $\pc$ with four vertices.}\label{MST}
				\end{figure}
			\end{example}
			
		\subsection{Optimality of Minimal Spanning $d$-Trees}\label{SUBSECTION: Construction and Optimality of a Minimal Spanning Tree}
		
			In this subsection we prove that the (final stage of the) $d$-forest constructed earlier is indeed a minimal spanning $d$-tree, that any minimal spanning tree can be obtained this way, and finally, that reduced versions of minimal spanning trees are optimal in the sense that they have minimal total weight among all $(d-1)$-fitting $d$-spanning $d$-forests in $\rwsx$ (that is, they are optimal at every scale, not just at the final one, as per definition). 
			
			\begin{lemma}\label{LEMMA: optimality of minimal spanning trees}
				For every $\alpha \in \RR$ the subcomplex $\armst$ is a $(d-1)$-fitting $d$-spanning $d$-forest in $\rwsx$, and moreover has minimal total weight among all $(d-1)$-fitting $d$-spanning $d$-forests in $\rwsx$.
			\end{lemma}
			
			\begin{proof}
				$\armst$ is a $d$-spanning $d$-forest by construction. As for the rest, it suffices to prove this for $\alpha \in \im(w) = \set{w_k}{k \in \intcc[\NN]{1}{n}}$. We prove it by induction on $k$. Certainly, this holds for $k = 0$ (as before, we use the notation $w_0 = -\infty$, $w_{n+1} = \infty$).
				
				Take $k \in \intcc[\NN]{1}{n}$ and assume $\armst[w_{k-1}]$ is a minimal $(d-1)$-fitting $d$-spanning $d$-forest. For fittingness it suffices to check that $\armst[w_k]$ is $(d-1)$-fitting in $\dsk[w_k]$. By Lemma~\ref{LEMMA: spanning subcomplex}(\ref{LEMMA: spanning subcomplex: lower-dimensional fittingness}), $\armst[w_k]$ is at least $(d-2)$-fitting and the map $\hm[d-1]{\armst[w_k] \hookrightarrow \dsk[w_k]}$ is surjective. To prove it is bijective, it suffices to verify that the dimensions of the domain and the codomain match.
				
				Using Lemma~\ref{LEMMA: spanning subcomplex}(\ref{LEMMA: spanning subcomplex: formula}) for $w_k$ and $w_{k-1}$ yields
				\begin{multline*}
					\big(\# \text{$d$-faces in $\rwsx[w_k]$}\big) + \beta_{d-1}(\dsk[w_k]) - \beta_d(\dsk[w_k]) \sepeq \\
					= \big(\# \text{$d$-faces in $\armst[w_k]$}\big) + \beta_{d-1}(\armst[w_k]) - \beta_d(\armst[w_k])
				\end{multline*}
				and
				\begin{multline*}
					\big(\# \text{$d$-faces in $\rwsx[w_{k-1}]$}\big) + \beta_{d-1}(\dsk[w_{k-1}]) - \beta_d(\dsk[w_{k-1}]) \sepeq \\
					= \big(\# \text{$d$-faces in $\armst[w_{k-1}]$}\big) + \beta_{d-1}(\armst[w_{k-1}]) - \beta_d(\armst[w_{k-1}]).
				\end{multline*}
				
				We know that $\armst[w_k]$ and $\armst[w_{k-1}]$ are $d$-forests, and the induction hypothesis tells us $\armst[w_{k-1}]$ is $(d-1)$-fitting, so the above equalities reduce to
				\[
					\big(\# \text{$d$-faces in $\rwsx[w_k]$}\big) + \beta_{d-1}(\dsk[w_k]) - \beta_d(\dsk[w_k]) \sepeq \big(\# \text{$d$-faces in $\armst[w_k]$}\big) + \beta_{d-1}(\armst[w_k]),
				\]
				\[
					\big(\# \text{$d$-faces in $\rwsx[w_{k-1}]$}\big) - \beta_d(\dsk[w_{k-1}]) \sepeq \big(\# \text{$d$-faces in $\armst[w_{k-1}]$}\big).
				\]
				
				Subtract these two equalities and rearrange the result to get
				\[
					\beta_{d-1}(\dsk[w_k]) - \beta_{d-1}(\armst[w_k]) =
				\]
				\begin{multline*}
					= \underbrace{\Big(\beta_d(\dsk[w_k]) - \beta_d(\dsk[w_{k-1}])\Big)}_{\# \text{$d$-faces with weight $w_k$ which increase $\beta_d$}} + \\
					+ \underbrace{\Big(\# \text{$d$-faces in $\armst[w_k]$} - \# \text{$d$-faces in $\armst[w_{k-1}]$}\Big)}_{\# \text{$d$-faces with weight $w_k$ which do not increase $\beta_d$}} - \\
					- \underbrace{\Big(\# \text{$d$-faces in $\rwsx[w_k]$} - \# \text{$d$-faces in $\rwsx[w_{k-1}]$}\Big)}_{\# \text{$d$-faces with weight $w_k$}}
				\end{multline*}
				which is zero, proving the desired equality of dimensions.
				
				We now prove minimality inductively on $k$. Clearly, the statement holds for $k = 0$.
				
				Let $\scx$ be a $(d-1)$-fitting $d$-spanning $d$-forest in $\rwsx[w_k]$. Define 
				\[
					\scx' \dfeq \set{F \in \scx}{w(F) < w_k}.
				\]
				Then $\scx'$ is a $d$-spanning $d$-forest in $\rwsx[w_{k-1}]$; in particular, $\hm[d-1]{\scx' \hookrightarrow \rwsx[w_{k-1}]}$ is surjective. Denote $m \dfeq \beta_{d-1}(\scx') - \beta_{d-1}(\rwsx[w_{k-1}])$. Using Lemma~\ref{LEMMA: spanning subcomplex}(\ref{LEMMA: spanning subcomplex: towards fittingness}) $m$ times, we get $d$-faces $F_1, \ldots, F_m \in \rwsx[w_{k-1}] \setminus \scx'$, such that $\scx'' \dfeq \scx' \cup \set{F_1, \ldots, F_m}$ is a $(d-1)$-fitting $d$-spanning $d$-forest in $\rwsx[w_{k-1}]$.
				
				By the induction hypothesis the total weight of $\armst[w_{k-1}]$ is at most the total weight of $\scx''$. Let $a \in \NN$ be the number of faces in $\wsx$ of dimension at most $d-1$ with weight $w_k$ and let $b \in \NN$ be the number of $d$-faces in $\scx$ of weight $w_k$. Then
				\[
					\tw(\scx) = \tw(\scx') + (a + b) \cdot w_k = \tw(\scx'') - \sum_{i = 1}^m w(F_i) + (a + b) \cdot w_k \geq
				\]
				\[
					\geq \tw(\scx'') + (a + b - m) \cdot w_k \geq \tw(\armst[w_{k-1}]) + (a + b - m) \cdot w_k = \tw(\armst[w_k]),
				\]
				where we still need to justify the final equality. That is, we need to check that we add $a + b - m$ faces when going from $\armst[w_{k-1}]$ to $\armst[w_k]$. Since $\armst$ is $d$-spanning at all times, this reduces to checking that $\armst[w_k]$ has $b - m$ more $d$-dimensional faces than $\armst[w_{k-1}]$.
				
				Refer again to Lemma~\ref{LEMMA: spanning subcomplex}(\ref{LEMMA: spanning subcomplex: formula}) to get
				\begin{multline*}
					\big(\# \text{$d$-faces in $\rwsx[w_k]$}\big) + \beta_{d-1}(\dsk[w_k]) - \beta_d(\dsk[w_k]) = \\
					= \big(\# \text{$d$-faces in $\scx$}\big) + \beta_{d-1}(\scx) - \beta_d(\scx),
				\end{multline*}
				\begin{multline*}
					\big(\# \text{$d$-faces in $\rwsx[w_{k-1}]$}\big) + \beta_{d-1}(\dsk[w_{k-1}]) - \beta_d(\dsk[w_{k-1}]) = \\
					= \big(\# \text{$d$-faces in $\scx'$}\big) + \beta_{d-1}(\scx') - \beta_d(\scx').
				\end{multline*}
				This reduces to
				\[
					\big(\# \text{$d$-faces in $\rwsx[w_k]$}\big) - \beta_d(\dsk[w_k]) = \big(\# \text{$d$-faces in $\scx$}\big),
				\]
				\[
					\big(\# \text{$d$-faces in $\rwsx[w_{k-1}]$}\big) + \beta_{d-1}(\dsk[w_{k-1}]) - \beta_d(\dsk[w_{k-1}]) = \big(\# \text{$d$-faces in $\scx'$}\big) + \beta_{d-1}(\scx').
				\]
				Hence
				\[
					m = \beta_{d-1}(\scx') - \beta_{d-1}(\dsk[w_{k-1}]) =
				\]
				\[
					= b + \Big(\beta_d(\dsk[w_k]) - \beta_d(\dsk[w_{k-1}])\Big) - \Big(\# \text{$d$-faces in $\rwsx[w_k]$} - \# \text{$d$-faces in $\rwsx[w_{k-1}]$}\Big),
				\]
				so
				\[
					b - m = \Big(\# \text{$d$-faces in $\rwsx[w_k]$} - \# \text{$d$-faces in $\rwsx[w_{k-1}]$}\Big) - \Big(\beta_d(\dsk[w_k]) - \beta_d(\dsk[w_{k-1}])\Big) 
				\]
				which by the calculation for $\armst$ in the fittingness part of the proof above equals
				\[
					\big(\# \text{$d$-faces in $\armst[w_k]$}\big) - \big(\# \text{$d$-faces in $\armst[w_{k-1}]$}\big).
				\]
			\end{proof}
			
			We claim that minimal spanning trees (as given by Definition~\ref{DEFINITION: Minimal spanning tree}) are precisely the complexes, obtained in Algorithm~\ref{ALGORITHM: Minimal spanning tree}, at their final stage.
			
			\begin{lemma}[Correctness of Algorithm~\ref{ALGORITHM: Minimal spanning tree}]\label{LEMMA: Algorithm for Minimal Spanning Trees}
				Let $\wsx$ be a weighted simplex and $\armst$ as given by Algorithm~\ref{ALGORITHM: Minimal spanning tree}.
				\begin{enumerate}
					\item\label{LEMMA: Algorithm for Minimal Spanning Trees: correctness of algorithm}
						For $\alpha \in \RR_{\geq w(\pc)}$ the complex $\armst$ is a minimal spanning $d$-tree of $\wsx$. Denoting $\mst \dfeq \armst[w(\pc)]$, we have $\rmst = \armst$ for all $\alpha \in \RR$.
					\item\label{LEMMA: Algorithm for Minimal Spanning Trees: universality}
						Every minimal spanning $d$-tree of $\wsx$ is of the form $\armst[w(\pc)]$, obtained via Algorithm~\ref{ALGORITHM: Minimal spanning tree}.
				\end{enumerate}
			\end{lemma}
			
			\begin{proof}
				\begin{enumerate}
					\item
						Use Lemma~\ref{LEMMA: optimality of minimal spanning trees} for $\alpha \geq w(\pc)$ while noting that in this case $\rwsx$ is the whole simplex, so has the homology of a point.
					\item
						Let $\mst$ be any minimal spanning tree. We get $\rmst = \armst$ for all $\alpha \in \RR$ if we choose the order of $d$-faces at any weight $w_k$ to start with the $d$-faces in $\mst$, followed by those not in $\mst$. It is clear from Algorithm~\ref{ALGORITHM: Minimal spanning tree} that $\armst$ includes all $d$-faces of weight $w_k$ in $\rmst$. To get the converse, note that $\armst$ and $\rmst$ (both of which are $(d-1)$-fitting $d$-spanning $d$-forests) have the same number of $d$-faces at every stage by Lemma~\ref{LEMMA: spanning subcomplex}(\ref{LEMMA: spanning subcomplex: formula}).
				\end{enumerate}
			\end{proof}
			
			\begin{remark}
				We conclude that Algorithm~\ref{ALGORITHM: Minimal spanning tree} yields a minimal spanning tree. The general idea of the algorithm was to take the necessary amount of $d$-faces in the tree (the exact number is given by Proposition~\ref{PROPOSITION: spanning trees}(\ref{PROPOSITION: spanning trees: counting faces})) while choosing first among lighter faces, so the greedy algorithm works (as one would anticipate from matroid theory).
			\end{remark}
			
			\begin{theorem}[Optimality of Minimal Spanning Trees]\label{THEOREM: Optimality of Minimal Spanning Trees}
				For any minimal spanning tree $\mst$ of a weighted simplex $\wsx$ and every $\alpha \in \RR$ the subcomplex $\rmst$ is a $(d-1)$-fitting $d$-spanning $d$-forest in $\rwsx$, and has a minimal total weight among all $(d-1)$-fitting $d$-spanning $d$-forests in $\rwsx$.
			\end{theorem}
			
			\begin{proof}
				By Lemma~\ref{LEMMA: Algorithm for Minimal Spanning Trees}(\ref{LEMMA: Algorithm for Minimal Spanning Trees: universality}) and Lemma~\ref{LEMMA: optimality of minimal spanning trees}.
			\end{proof}

	\section{Homologically Persistent $d$-Skeleton}\label{SECTION: homologically persistent skeleton}
		
		We proved in Theorem~\ref{THEOREM: Optimality of Minimal Spanning Trees} that homology of the minimal spanning $d$-tree matches the homology of a weighted simplex up to dimension $d-1$ for all parameter values. The purpose of the homologically persistent skeleton is to add and remove $d$-faces, called critical $d$-faces, in a way that ensures an isomorphism of homology groups in dimension $d$ as well.
		
		\subsection{Critical Faces of a Weighted Simplex}
			
			Fix a minimal spanning $d$-tree $\mst$ of a weighted simplex $\wsx$.
			
			\begin{definition}
			\label{def: critical_face}
				A $d$-face $K$ of $\wsx$ is \df{critical} when $K$ is not in $\mst$.
			\end{definition}
			
			In order to obtain isomorphisms on the level of homology in Theorem~\ref{THEOREM: fittingness and optimality of reduced skeleton}, critical faces play a crucial role as generators of homology (at all stages $\alpha \in \RR$). However, a critical face might contribute to many nontrivial cycles, so the connection between critical $d$-faces and generators in $\hm[d]{\rwsx}$ is not canonical in general. We resolve this issue by using relative homology.
			
			\begin{lemma}\label{LEMMA: critical faces and relative homology}
				Let $\alpha \in \RR$ and let $\scx$ be a subcomplex of $\rwsx$ which contains $\rmst$. Let $K_1, K_2, \ldots, K_m$ be the critical $d$-faces in $\scx$.
				\begin{enumerate}
					\item\label{LEMMA: critical faces and relative homology: relative homology classes}
						Each $K_i$ represents a relative homology class $[K_i] \in \hm[d]{\scx, \rmst}$.
					\item\label{LEMMA: critical faces and relative homology: generators}
						The classes $[K_1], [K_2], \ldots, [K_m]$ generate $\hm[d]{\scx, \rmst}$.
					\item\label{LEMMA: critical faces and relative homology: basis}
						If $\scx$ is a $d$-complex, the classes $[K_1], [K_2], \ldots, [K_m]$ form a basis of $\hm[d]{\scx, \rmst}$.
				\end{enumerate}
			\end{lemma}
			
			\begin{proof}
				\begin{enumerate}
					\item
						By assumption $K_i$ is in $\scx$. The boundary of $K_i$ is in the $(d-1)$-skeleton of $\rwsx$ and thus also in $\rmst$, meaning that $K_i$ is a relative $d$-cycle. Hence $[K_i] \in \hm[d]{\scx, \rmst}$.
					\item
						Take any $[z] \in \hm[d]{\scx, \rmst}$. We write $z = \sum_{i} c_i F_i$, where $c_i \in \field$ and $F_i$s are $d$-faces of $\scx$. Whenever $F_i$ is in the minimal spanning tree, $[F_i] = 0$ in $\hm[d]{\scx, \rmst}$, which implies that $[z] = \sum_{i, F_i \notin \rmst} c_i [F_i]$. The class $[z]$ can therefore be expressed as a linear combination of classes represented by critical faces.
					\item
						Suppose we have $\sum_{i = 1}^m c_i [K_i] = 0$; then $\big[\sum_{i = 1}^m c_i K_i\big] = 0$. This means there exist $v \in \ch[d+1]{\scx}$ and $u \in \ch[d]{\rmst}$ such that
						\[
							\partial_{d+1}{v} = u + \sum_{i = 1}^m c_i K_i.
						\]
						But as a $d$-complex, $\scx$ only has $0$ as a $(d+1)$-chain, so we get
						\[
							u + \sum_{i = 1}^m c_i K_i = 0.
						\]
						Write $u = \sum_{j = 1}^k d_j F_j$ where $F_j$s are $d$-faces in $\rmst$. Thus
						\[
							\sum_{j = 1}^k d_j F_j + \sum_{i = 1}^m c_i K_i
						\]
						is the zero chain, and since $d$-faces form a basis of the space of $d$-chains, all the coefficients must be zero. In particular, $c_i$s are zero, proving the desired linear independence.
				\end{enumerate}
			\end{proof}
			
			To build the homologically persistent skeleton associated to $\wsx$, we add and remove critical $d$-faces, to which we assign birth and death times inductively (Subsection~\ref{SECTION: Critical Faces}). The following lemma guarantees that for $d > 0$ all homology classes generated by critical faces eventually die.
			
			\begin{lemma}\label{LEMMA: late relative homology}
				For any $\alpha \in \RR_{\geq w(\pc)}$ we have
				\[
					\hm[d]{\rwsx, \rmst} \ism \hm[d]{\wsx, \mst} \ism \hm[d]{\wsx} \ism
					\begin{cases}
						0 & \text{if } d > 0, \\
						\field & \text{if } d = 0.
					\end{cases}
				\]
			\end{lemma}
			
			\begin{proof}
				The first isomorphism is obvious. So is the last one, since $\wsx$ is contractible (it is a simplex). As for the middle one, consider first the case $d \geq 2$. In the long exact sequence of a pair
				\[
					\ldots \to \hm[d]{\mst} \to \hm[d]{\wsx} \to \hm[d]{\wsx, \mst} \to \hm[d-1]{\mst} \to \ldots
				\]
				we have $\hm[d]{\mst} = \hm[d-1]{\mst} = 0$, so we get the desired isomorphism. If $d = 1$, the map $\hm[0]{\mst} \to \hm[0]{\wsx}$, induced by the inclusion, is an isomorphism $\field \ism \field$, so the boundary map ${\hm[1]{\wsx, \mst} \to \hm[0]{\mst}}$ is zero. Hence ${\hm[1]{\wsx} \to \hm[1]{\wsx, \mst}}$ is surjective. It is also injective since $\hm[1]{\mst} = 0$. If $d = 0$, then $\mst = \emptyset$, so ${\hm[0]{\wsx} \ism \hm[0]{\wsx, \mst}}$.
			\end{proof}
		
		\subsection{Birth and Death of a Critical Face}\label{SECTION: Critical Faces}
			
			For each critical $d$-face we define its \df{birth time} (or simply \df{birth}) to be its weight. We wish to define the death time of a critical face as the parameter value at which the homology generator it created dies, however, it can happen that multiple critical faces enter at the same time. In that case assigning death times correctly is critical for Theorem~\ref{THEOREM: fittingness and optimality of reduced skeleton} to hold.
			
			\begin{example}\label{Example:Kurlin}
				Consider the point cloud depicted in Figure~\ref{FIGURE: fixed death times} for $d = 1$. The only two critical $1$-faces, which do not immediately die, are depicted in red (they appear at the parameter value $1$). The generators they create die at times $\frac{5}{4}$ and $\frac{\sqrt{13}}{2}$, but since they are born at the same time, the question that arises is which death time to associate to which critical face. It turns out that for Theorem~\ref{THEOREM: fittingness and optimality of reduced skeleton} to hold, the choice of assignments in Figure~\ref{FIGURE: unique death time} is the only valid one. However, that does not mean that we never have any freedom of assigning death times. A minor change in the example (see Figure~\ref{FIGURE: arbitrary death times}) allows us two possibilities, both valid for Theorem~\ref{THEOREM: fittingness and optimality of reduced skeleton}.

				\begin{figure}[!ht]
					\centering
					\begin{tabular}{ccccccc}
						\includegraphics[scale=0.57]{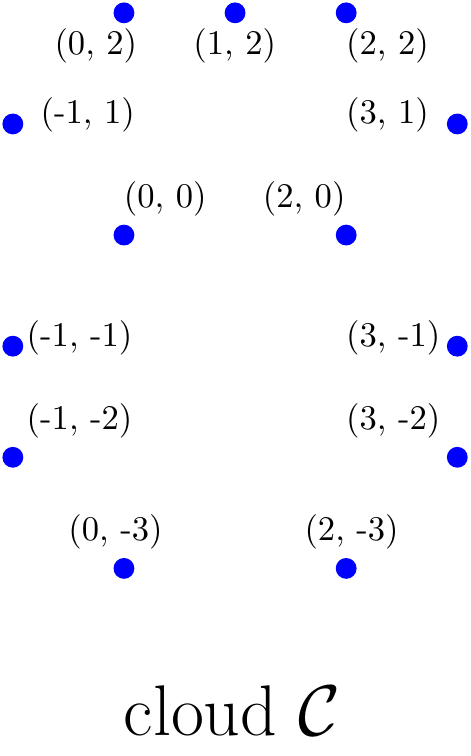} && \quad &
						\includegraphics[scale=0.57]{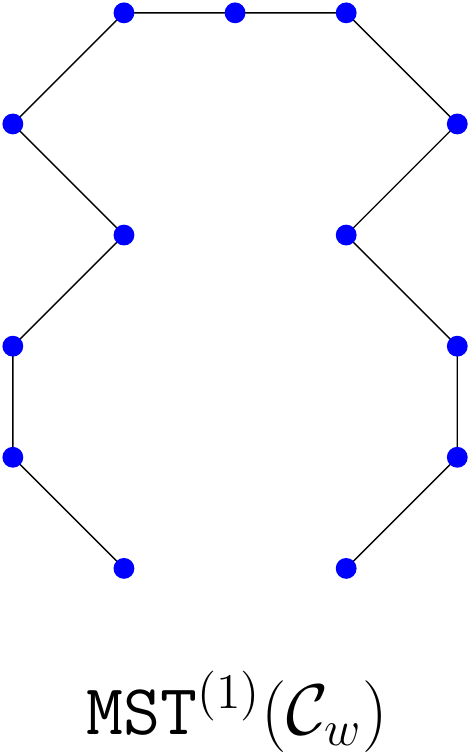} && \quad &
						\includegraphics[scale=0.57]{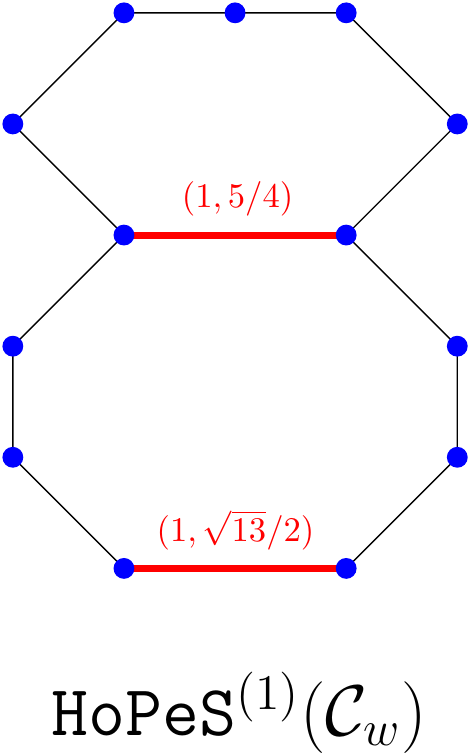}
					\end{tabular}
					\caption{Cloud $\pc$ whose simplex $\wsx$ has weights from its \v{C}ech complex, its minimal spanning 1-tree and its homologically persistent $1$-skeleton.}\label{FIGURE: fixed death times}\label{FIGURE: unique death time}
				\end{figure}
				
				\begin{figure}[!ht]
					\centering
				\hspace*{-1.75cm}	\begin{tabular}{ccccccc}
						\includegraphics[scale=0.57]{cloudcoord.pdf} && \quad &
						\includegraphics[scale=0.57]{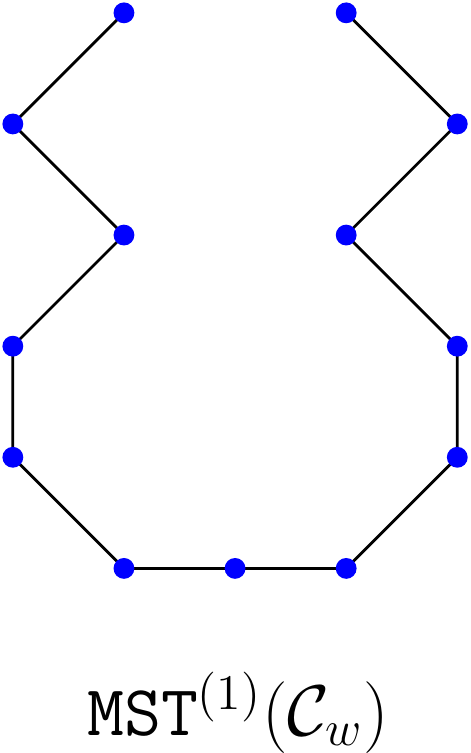} && \quad &
						\includegraphics[scale=0.57]{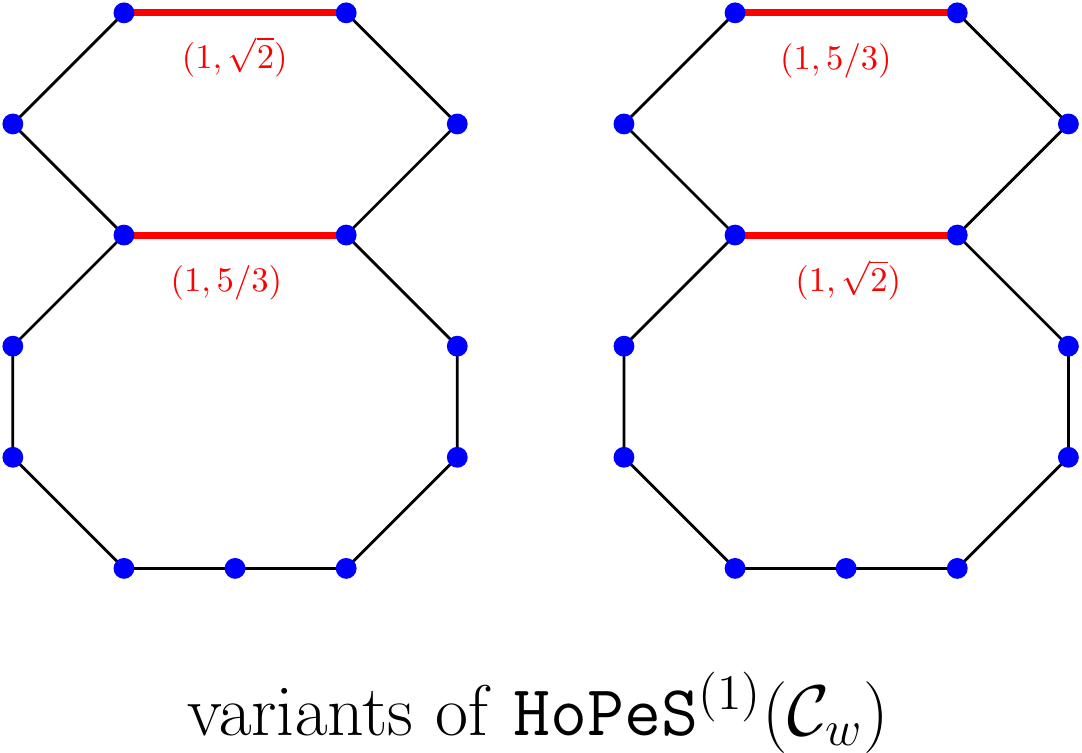}
					\end{tabular}
					\caption{Cloud $\pc$ whose simplex $\wsx$ has weights from its \v{C}ech complex, its minimal spanning $1$-tree and two possible homologically persistent $1$-skeleta.}\label{FIGURE: arbitrary death times}
				\end{figure}
			\end{example}
			
			As Example~\ref{Example:Kurlin} shows, we need to know when and to what extent the assignment of death times is determined. We describe an algorithm, which assigns death times to all critical $d$-faces and determines exactly how much freedom we have for these assignments.
			
			Deaths can only occur at times when a simplex is added to $\rwsx$, \ie at values in the image $\im(w)$. We go through $\im(w)$ with $\alpha$ in increasing order and for each such $\alpha \in \im(w)$ decide which (if any) critical $d$-faces die at $\alpha$.
			
			\begin{definition}[Deaths of Critical Faces]\label{def:deaths_critical_faces}
				Define $\crit \dfeq \set{K_1, K_2, \ldots, K_s}$ to be the set of critical $d$-faces born before or at $\alpha$ that have not yet been assigned a death time. By Lemma~\ref{LEMMA: critical faces and relative homology} the classes $[K_1], [K_2], \ldots, [K_s]$ form a basis of $\hm[d]{(\rmst \cup \crit, \rmst)}$. Denote
				\[
					f \dfeq \hm[d]{(\rmst \cup \crit, \rmst) \hookrightarrow (\rwsx, \rmst)}
				\]
				and set $r \dfeq \dim\ker(f)$. Choose a basis $\set{b_1, b_2, \ldots, b_r}$ of $\ker(f)$. We can expand each basis vector as
				\[
					b_i = \sum_{j = 1}^s c_{ij} [K_j]
				\]
				with $c_{ij} \in \field$. Consider the system of equations (in the field $\field$)
				\[
					\sum_{j = 1}^s c_{ij} x_j = 0
				\]
				for $i \in \intcc[\NN]{1}{r}$. Since basis elements are linearly independent, so are these equations. Thus there are $r$ leading variables, for which the system may be solved, expressing them with the remaining $s - r$ free variables. Let $I \subseteq \intcc[\NN]{1}{s}$ be the set (possibly empty, if $\ker(f)$ is trivial) of indices of leading variables. 
				For each $i \in I$ we declare that the \df{death time} of the critical face $K_i$ is $\alpha$.
			
				Depending on the system of equations, we might have many possible choices, which variables to choose as the leading ones. No further restriction on this choice is necessary for Theorem~\ref{THEOREM: fittingness and optimality of reduced skeleton}(\ref{THEOREM: fittingness and optimality of reduced skeleton: fittingness}) (fittingness), but to get the rest of the theorem (optimality), we need to further insist on the \df{elder rule} (compare with the elder rule for the construction of the persistence diagram~\cite[page~151]{eharer}): among all available choices for the set of leading variables, choose the one with the largest total weight. There may be more than one set of possible leading variables with the maximal total weight --- this is the amount of freedom we have when choosing death times.
				
				If $d \geq 1$, this process assigns death times to all critical $d$-faces: if any are still left at $\alpha = w(\pc)$, they all die at that time since $\hm[d]{\rwsx[w(\pc)], \rmst[w(\pc)]} = 0$ by Lemma~\ref{LEMMA: late relative homology}. However, if $d = 0$, $\hm[d]{\rwsx[w(\pc)], \rmst[w(\pc)]}$ is $1$-dimensional rather than $0$-dimensional. As such, we declare the death time of the final critical $0$-face to be $\infty$. This makes sense: critical $0$-faces (\ie vertices) die as the complex becomes more and more connected, but in the end a single connected component endures indefinitely.
			\end{definition}
			
			Here is the summary of this procedure, given as an explicit algorithm.
			\begin{algorithm}[H]
			\label{alg:deaths}
				\caption{Death times of critical $d$-faces}\label{ALGORITHM: Death times of critical faces}
				\begin{algorithmic}[1]
					\STATE $\death(K) \dfeq \infty$ for all $K \in \skl[d]{\wsx} \setminus \mst$
					\STATE $w_1, w_2, \ldots, w_n \dfeq$ elements of $\im(w)$, in order
					\FOR{$l = 1$ \TO $n$}
						\STATE $\crit \dfeq \set{K_1, K_2, \ldots, K_s} \dfeq \set[1]{K \in \skl[d]{\rwsx[w_l]} \setminus \rmst[w_l]}{\death(K) = \infty}$
						\STATE $f \dfeq \hm[d]{(\rmst[w_l] \cup \crit, \rmst[w_l]) \hookrightarrow (\rwsx[w_l], \rmst[w_l])}$
						\STATE $\set{b_1, b_2, \ldots, b_r} \dfeq$ a choice of a basis of $\ker(f)$
						\FOR{$i = 1$ \TO $r$}
							\FOR{$j = 1$ \TO $s$}
								\STATE $c_{ij} \dfeq$ coefficient at $[K_j]$ in the expansion of $b_i$
							\ENDFOR
						\ENDFOR
						\STATE $I \dfeq$ a choice of an $r$-element subset of $\intcc[\NN]{1}{s}$, such that
							\begin{itemize}[noitemsep, topsep=0pt]
								\item\hspace{-3ex}
									the system $\big(\sum_{j = 1}^s c_{ij} x_j = 0\big)_{i \in \intcc[\NN]{1}{r}}$ is solvable on variables $\set{x_j}{j \in I}$,
								\item\hspace{-3ex}
									the total weight of $\set{K_j}{j \in I}$ is maximal among such subsets
							\end{itemize}
						\STATE $\death(K_j) \dfeq w_l$ for all $j \in I$
					\ENDFOR
				\end{algorithmic}
			\end{algorithm}
			
			For any critical $d$-face $K$ define its \df{lifespan} to be $\death(K) - \birth(K)$. It is possible for a critical $d$-face $K$ to have the lifespan $0$, if the homology class $[K]$ gets killed by some $(d+1)$-face(s) that have the same weight as $K$.
			
			\begin{lemma}\label{LEMMA: critical basis}
				For any $\alpha \in \RR$ define
				\[
					\critl \dfeq \set{K \text{ critical $d$-face}}{\textrm{birth}(K) \leq \alpha < \textrm{death}(K)}.
				\]
				The classes, represented by faces in $\critl$, form a basis of $\hm[d]{\rwsx, \rmst}$.
			\end{lemma}
			
			\begin{proof}
				It suffices to check this for $\alpha \in \im(w)$. We know that $[K]$s with ${\birth(K) \leq \alpha}$ generate $\hm[d]{\rwsx, \rmst}$ by Lemma~\ref{LEMMA: critical faces and relative homology}. We need to check that $[K]$ represented by a critical face, which is dead at $\alpha$, can be expressed by those still living at $\alpha$. We prove this inductively on decreasing death times. Let $\delta \leq \alpha$ be the death time of $K$. By Definition~\ref{def:deaths_critical_faces} we can write 
				\[
					[K] = \sum_{j=1}^s c_{j}[K_j]
				\]
				in $\hm[d]{\rwsx[\delta], \rmst[\delta]}$, where $\critl[\delta] = \set{K_1, K_2, \ldots, K_s}$, \ie death times of $K_j$ are larger than $\delta$. By applying $\hm[d]{(\rwsx[\delta], \rmst[\delta]) \hookrightarrow (\rwsx, \rmst)}$ we can see that this equation also holds in $\hm[d]{\rwsx, \rmst}$. By the induction hypothesis all of these $[K_j]$ can be expressed by the still living critical faces and therefore, so can $[K]$.
				
				As for linear independence, redefine $K_1, \ldots, K_s$ to be all the faces in $\critl$. Assume that $\sum_{j = 1}^s c_j [K_j] = 0$ in $\hm[d]{\rwsx, \rmst}$. This implies that 
				\[
					\sum_{j = 1}^s c_j [K_j] \in \ker\hm[d]{(\rmst \cup \crit, \rmst) \hookrightarrow (\rwsx, \rmst)}.
				\]
				By assumption none of $[K_j]$s die at $\alpha$, so this kernel is trivial, meaning $\sum_{j = 1}^s c_j [K_j] = 0$ in $\hm[d]{\rmst \cup \crit, \rmst}$. Since $[K_j]$s form a basis of this homology (Lemma~\ref{LEMMA: critical faces and relative homology}), the coefficients $c_j$ are zero.
			\end{proof}
			
		\subsection{Optimality of a Homologically Persistent $d$-Skeleton}\label{SUBSECTION: fittingness and optimality of reduced skeleton}
			
			We continue following the blueprint from~\cite{K15} where the homologically persistent $1$-skeleton was constructed by taking a minimal spanning $1$-tree and adding labeled critical edges. However, we find it more convenient to have all simplices in the homologically persistent skeleton to be of the same type, so we shall label \emph{all} faces. Define a \df{label} to be a pair $(l, r) \in \eR \times \eR$ such that $0 \leq l < r$. Call $l$ the \df{left label} and $r$ the \df{right label}.
			
			\begin{definition}[homologically persistent skeleton]\label{def:HoPeS}
				Given $d \in \NN$ and a weighted simplex $\wsx$, its \df{homologically persistent $d$-skeleton} $\hps(\wsx)$ is the (choice of a) minimal spanning $d$-tree together with all critical $d$-faces with positive lifespan:
				\[
					\hps(\wsx) \dfeq \mst \cup \set{K \in \skl[d]{\wsx} \setminus \mst}{\birth(K) < \death(K)}.
				\]
				Each face $F$ in $\hps(\wsx)$ is labeled: if $F$ is in $\mst$, by $\big(w(F), \infty\big)$; otherwise by $\big(\birth(F), \death(F)\big)$. We write simply $\hps$ instead of $\hps(\wsx)$ when there is no ambiguity.
			\end{definition}
			
			Note that the set of labels $\set{(l, r) \in \eR \times \eR}{0 \leq l < r}$ can be seen as a form of an interval~domain~\cite{3720}. In particular, we have the \df{information order} $\inford$, given by
			\[
				(l', r') \inford (l'', r'') \sepdfeq l' \leq l'' \ \land \ r' \geq r''.
			\]
			Labeling of $\hps$ is monotone in the following sense. Let $F$ and $G$ be faces in $\hps$ with labels $\ell_F$ and $\ell_G$ respectively. If $F \subseteq G$, then $\ell_F \inford \ell_G$.
			
			This means that $\hps$ is a kind of a `weighted complex' itself --- except that instead of the weighting mapping into $\RR_{\geq 0}$ with its usual order $\leq$, it maps into the interval domain of labels, equipped with the information order. The consequence is that we can define the reduced version of the homologically persistent skeleton for any $\alpha \in \RR$:
			\[
				\rhps(\wsx) \dfeq \set[1]{\big(F, (l, r) \big) \in \hps(\wsx)}{l \leq \alpha < r}.
			\]
			As usual, we shorten $\rhps(\wsx)$ to $\rhps$ when there is no ambiguity. Due to monotonicity of labeling, $\rhps$ is a (labeled) simplicial complex.
			
			\begin{example}\label{Hopesexample}
				Let $\pc$ be a point cloud from Example~\ref{MSTexample}. 
				To make pictures understandable for readers, Figure~\ref{Hopes} shows the rather small, low-dimensional complexes $\rhps$, which are accidentally similar to $\wsx$ for most $\alpha$. 
				\begin{figure}[!ht]
					\centering
					\hspace*{-0.5cm}\includegraphics{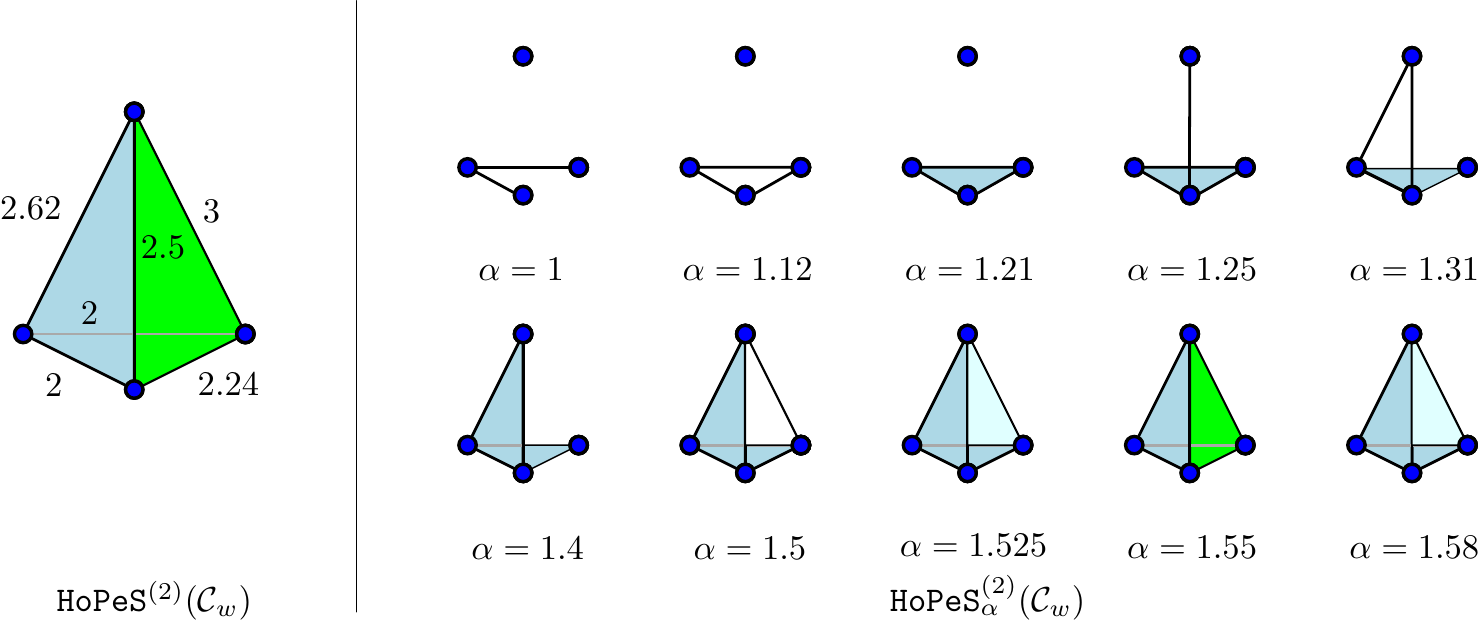}
					\caption{Geometric realizations of $\mathtt{HoPeS}^{(2)}(\wsx)$ and its reduced versions with respect to \v{C}ech filtration of a point cloud $\pc$ with four vertices. $\mathtt{HoPeS}^{(2)}(\wsx)$ consists of the boundary of the tetrahedron, its only critical face marked by green. The remaining $2$-faces are a part of the minimal spanning tree (\cf Figure~\ref{MST}).}\label{Hopes}
				\end{figure}
			\end{example}
			
			\begin{lemma}\label{LEMMA: d isomorphism}
				$\hm[d]{\rhps \hookrightarrow \rwsx}$ is an isomorphism for any $\alpha \in \RR$.
			\end{lemma}
			
			\begin{proof}
				By Lemma~\ref{LEMMA: spanning subcomplex}(\ref{LEMMA: spanning subcomplex: connecting relative and absolute homology}) the map $\hm[d]{\rhps \hookrightarrow \rwsx}$ is an isomorphism if and only if the map $\hm[d]{(\rhps, \rmst) \hookrightarrow (\rwsx, \rmst)}$ is. But that follows immediately from Lemma~\ref{LEMMA: critical faces and relative homology}(\ref{LEMMA: critical faces and relative homology: basis}) and Lemma~\ref{LEMMA: critical basis}.
			\end{proof}
			
			\begin{lemma}\label{LEMMA: d-fitting subcomplexes of the reduced weighted simplex}
				Take any $\alpha \in \RR$ and any $d$-fitting $d$-spanning $d$-subcomplex $\scx$ in $\rwsx$. By Lemma~\ref{LEMMA: spanning subcomplex}(\ref{LEMMA: spanning subcomplex: existence of a nice subforest}) $\scx$ contains a $d$-subcomplex which is a $(d-1)$-fitting $d$-spanning $d$-forest; let $F$ denote one with minimal total weight.
				\begin{enumerate}
					\item\label{LEMMA: d-fitting subcomplexes of the reduced weighted simplex: number od d-faces}
						The number of $d$-faces in $\scx$ is
						\[
							\Big(\# \text{$d$-faces in $\rwsx$}\Big) - \beta_d(\rwsx^{(d)}) + \beta_d(\rwsx).
						\]
						The number of $d$-faces in $F$ is
						\[
							\Big(\# \text{$d$-faces in $\rwsx$}\Big) - \beta_d(\rwsx^{(d)}).
						\]
						Consequently, the number of $d$-faces in $\scx \setminus F$ is equal to $\beta_d(\rwsx)$.
					\item\label{LEMMA: d-fitting subcomplexes of the reduced weighted simplex: isomorphisms}
						The diagram
						\[
							\begin{tikzpicture}[baseline=(current  bounding  box.center), scale = 1]
								\node (Ta) at (5, 2) {$ \hm[d]{\scx}$};
								\node (Tb) at (5,0) {$\hm[d]{\rwsx}$};
								\node (Tc) at (8.5, 2) {$\hm[d]{\scx, F}$};
								\node (Td) at (8.5, 0) {$\hm[d]{\rwsx, F}$};
								\path[->]
								(Ta) edge node[left] {{\small $\hm[d]{\scx \hookrightarrow \rwsx}$}} (Tb)
								(Tc) edge node[right] {{\small $\hm[d]{(\scx, F) \hookrightarrow (\rwsx, F)}$}} (Td)
								(Ta) edge node[above, yshift=0.5cm] {{\small $\hm[d]{(\scx, \emptyset) \hookrightarrow (\scx, F)}$}} (Tc)
								(Tb) edge node[below, yshift=-0.5cm] {{\small $\hm[d]{(\rwsx, \emptyset) \hookrightarrow (\rwsx, F)}$}} (Td);
							\end{tikzpicture}
						\]
						commutes and all maps in it are isomorphisms.
					\item\label{LEMMA: d-fitting subcomplexes of the reduced weighted simplex: weight bound}
						The diagram in the previous item induces a bijective correspondence between the set of $d$-faces in $\scx \setminus F$ and the set of dots $(p, q)$ in the persistence diagram $\pdwsx$ with $p \leq \alpha < q$. If a $d$-face $S$ is associated to the dot $(p, q)$, then $p \leq w(S)$.
				\end{enumerate}
			\end{lemma}
			
			\begin{proof}
				\begin{enumerate}
					\item
						Apply Lemma~\ref{LEMMA: spanning subcomplex}(\ref{LEMMA: spanning subcomplex: formula}) for $\scx$ and $F$ (as subcomplexes of $\rwsx$) and take their properties into account.
					\item
						Use Lemma~\ref{LEMMA: spanning subcomplex}(\ref{LEMMA: spanning subcomplex: connecting relative and absolute homology}) and the assumption that $\scx$ is $d$-fitting in $\rwsx$.
					\item
						Denote
						\begin{align*}
							f &&\dfeq& &&\hm[d]{\scx \hookrightarrow \rwsx} \ \circ \ \big(\hm[d]{(\scx, \emptyset) \hookrightarrow (\scx, F)}\big)^{-1} = \\
							&&=& &&\big(\hm[d]{(\rwsx, \emptyset) \hookrightarrow (\rwsx, F)}\big)^{-1} \ \circ \ \hm[d]{(\scx, F) \hookrightarrow (\rwsx, F)};
						\end{align*}
						this is an isomorphism between $\hm[d]{\scx, F}$ and $\hm[d]{\rwsx}$ by the previous item. Let $S_1, \ldots, S_m$ be $d$-faces in $\scx \setminus F$. By a similar argument as in Lemma~\ref{LEMMA: critical faces and relative homology} the classes $[S_i]$ form a basis of $\hm[d]{\scx, F}$. Hence $f([S_i])$ form a basis of $\hm[d]{\rwsx}$ and are thus in bijective correspondence with dots $(p, q)$ in $\pdwsx$ with $p \leq \alpha < q$. Let us denote the dot, associated to $S_i$, by $(p_i, q_i)$.
						
						Since $F$ has minimal total weight, $S_i$ has the largest weight among faces (with non-zero coefficients) in the cycle which represents $f([S_i])$. Hence the homology class $f([S_i])$ could not be born after $w(S_i)$.
				\end{enumerate}
			\end{proof}
			
			\begin{theorem}[Fittingness and Optimality of Reduced $d$-Skeletons]\label{THEOREM: fittingness and optimality of reduced skeleton}
				The following holds for every weighted simplex $\wsx$, $d \in \NN$, and $\alpha \in \RR$.
				\begin{enumerate}
					\item\label{THEOREM: fittingness and optimality of reduced skeleton: fittingness}
						$\rhps$ is $d$-fitting in $\rwsx$.
					\item
						For every critical $d$-face $K$ in $\rhps$, the dot $(p, q)$ in the persistence diagram, associated to it via the bijective correspondence from Lemma~\ref{LEMMA: d-fitting subcomplexes of the reduced weighted simplex}(\ref{LEMMA: d-fitting subcomplexes of the reduced weighted simplex: weight bound}) (for $\scx = \rhps$ and $F = \rmst$), is the same as the label of $K$. In particular $p = w(K)$.
					\item
						$\rhps$ has the minimal total weight among all $d$-fitting $d$-spanning subcomplexes $\scx \subseteq \rwsx$.
				\end{enumerate}
			\end{theorem}
			
			\begin{proof}
				\begin{enumerate}
					\item
						Between Lemmas~\ref{LEMMA: spanning subcomplex}(\ref{LEMMA: spanning subcomplex: lower-dimensional fittingness}) and~\ref{LEMMA: d isomorphism} we only still need to check that the map $\hm[d-1]{\rhps \hookrightarrow \rwsx}$ is injective, or equivalently, that its kernel is trivial.
						
						Let $K_1, K_2, \ldots, K_s$ be critical $d$-faces living at $\alpha$ and let $K_{s+1}, \ldots, K_m$ be the remaining critical $d$-faces born before or at $\alpha$. Take such a cycle $z \in \cy[d-1]{\rhps}$ that $[z] = 0$ in $\hm[d-1]{\rwsx}$. This means there exists a chain $v \in \ch[d]{\rwsx}$ with $\partial_d{v} = z$. Write
						\[
							v = \sum_{i = 1}^m c_i K_i + u
						\]
						where $u \in \ch[d]{\rmst}$. Using Lemma~\ref{LEMMA: critical basis} and unpacking relative homology we can express each $K_i$ with $i > s$ as
						\[
							K_i = \Big(\sum_{l = 1}^s e_l K_l\Big) + u_i + \partial_{d+1}{t_i}
						\]
						where $u_i \in \ch[d]{\rmst}$ and $t_i \in \ch[d+1]{\rwsx}$. Hence
						\[
							v = \sum_{i = 1}^s c'_i K_i + u' + \partial_{d+1}{t'}
						\]
						for suitable $c'_i \in \field$, $u' \in \ch[d]{\rmst}$ and $t' \in \ch[d+1]{\rwsx}$. Set
						\[
							v' \dfeq \sum_{i = 1}^s c'_i K_i + u',
						\]
						so $v' \in \ch[d]{\rhps}$. Then
						\[
							\partial_d{v'} = \partial_d{v'} + \partial_d\partial_{d+1}{t'} = \partial_d{v} = z.
						\]
						We conclude that $[z] = 0$ in $\hm[d-1]{\rhps}$.
					\item
						Let $K_1, \ldots, K_m$ be critical $d$-faces in $\rhps$ and for each $K_i$ let $(p_i, q_i)$ be the dot in the persistence diagram $\pdwsx$, associated to it. By the assignment of birth and death times of critical faces, as well as the previous item, we see that a cycle representing a homology class associated to $K_i$ (the birth of which is $p_i$) is born exactly at the time $K_i$ appeared in the homologically persistent skeleton, \ie at $w(K_i)$. 
					\item
						Let $S_1, \ldots, S_m$ be $d$-faces in $\scx \setminus F$ and for each $S_i$ let $(p_i, q_i)$ be the dot in the persistence diagram $\pdwsx$, associated to it; we have $p_i \leq w(S_i)$ (Lemma~\ref{LEMMA: d-fitting subcomplexes of the reduced weighted simplex}). Taking into account the previous item, we conclude
						\[
							\tw(\scx) = \tw(F) + \sum_{i = 1}^m w(S_i) \geq \tw(F) + \sum_{i = 1}^m p_i \geq
						\]
						\[
							\geq \tw(\rmst) + \sum_{i = 1}^m p_i = \tw(\rmst) + \sum_{i = 1}^m w(K_i) = \tw(\rhps).
						\]
				\end{enumerate}
			\end{proof}

	\section{Conclusion}
		
		We introduced a $d$-dimensional homologically persistent skeleton solving the Skeletonization Problem from Subsection~\ref{sub:motivations} in an arbitrary dimension $d$.
		\begin{itemize}
			\item
				Given a filtration of complexes on a point cloud $\pc$, Theorem~\ref{THEOREM: Optimality of Minimal Spanning Trees}(3) proves the optimality of minimal spanning $d$-trees of the cloud $\pc$.
			\item
				Definition~\ref{def:HoPeS} introduces $\hps$ by adding to a minimal spanning $d$-tree all critical $d$-faces that represent persistent homology $d$-cycles of $\wsx$, hence $\hps$ visualizes the persistence directly on data.
			\item
				For any scale $\alpha$ by Theorem~\ref{THEOREM: fittingness and optimality of reduced skeleton} the full skeleton $\hps$ contains a reduced subcomplex $\rhps$, which has a minimal total weight among all $d$-subcomplexes containing $\skl[d-1]{\rwsx}$ such that the inclusion into $\rwsx$ induces isomorphisms in homology in all degrees up to $d$.
		\end{itemize}
		
		The independence of the Euler characteristic from homology coefficients has helped to prove all results for homology over an arbitrary field $\field$. 
		Will Theorems~\ref{THEOREM: Optimality of Minimal Spanning Trees} and~\ref{THEOREM: fittingness and optimality of reduced skeleton} hold over an arbitrary unital commutative ring $R$? The answer is no, at least not in the form as they are currently stated. Assume that the theorems hold for $R$. Note that the proof of Lemma~\ref{LEMMA: critical faces and relative homology} works for a general $R$, so $\hm[d]{\rwsx}{R} \ism \hm[d]{\rhps}{R} \ism \hm[d]{\rhps, \rmst}{R}$ are free $R$-modules. That is, the results can only work if the homology over $R$ of every finite simplicial complex in every dimension is free. This of course excludes all the usual non-field homology coefficients, including $\ZZ$.
				
Similarly to a minimal spanning tree, a higher-dimensional $\hps$ is arborescent and will require some pruning and fine optimization as was done in Figures 6-8 from \cite{K15}.
If a subcomplex $\hps_{\alpha}$ is considered for a fixed scale $\alpha$, one could remove higher-dimensional `branches' with a total weight less than $\alpha$.
If one needs to avoid a manual choice of $\alpha$, one can define derived subskeletons of $\hps$ as in Definition 14 of \cite{K15} that contain only critical faces contributing to cycles with a high enough persistence.
	
Algorithms~\ref{ALGORITHM: Minimal spanning tree} and \ref{alg:deaths} are based on standard linear algebra, hence in the worst case cubic in the size of a given simplicial complex.
Such complexes for real data are not arbitrary and linear algebra computations usually scale almost linearly in the size of a complex. 
	
	We have implemented an algorithm computing the homologically persistent skeleton in Mathematica. We look forward to finding collaborators interested in implementing the code in C++ and trying it out on real data. We intend to discuss these implementations and the computational complexity of algorithms in a follow-up paper.

Finally, we thank all reviewers of this paper for their valuable time and helpful suggestions. 
		
	\printbibliography
	
\end{document}